\documentclass{article}
\usepackage[margin=35mm]{geometry}
\usepackage{amsthm,amsmath,amssymb}
\usepackage[pdf]{pstricks} 

\usepackage{upgreek}
\usepackage{mathtools}  
\usepackage{tabularx}   
\usepackage{xfrac}
\usepackage{tikz}
\usepackage{pgfplots}
\usetikzlibrary{matrix,shapes,arrows,positioning,shadows,backgrounds}

\newtheorem{intro}{Theorem}
\newenvironment{introtheorem}[1]
  {\renewcommand\theintro{#1}\intro}
  {\endintro}

\newcommand{\eqdef}{\coloneqq}

\newcommand{\C}{\mathbb{C}}
\newcommand{\Z}{\mathbb{Z}}
\newcommand{\N}{\mathbb{N}}

\newcommand{\NI}{\overline{\mathbb{N}}_+}

\def\bl{\hspace{0.5pt}\underline{\hphantom{\hspace{0.6em}}}\hspace{0.5pt}}	

\usepackage{hyperref}

\newtheorem*{theorem*}{Theorem}
\newtheorem{theorem}{Theorem}
\newtheorem{proposition}[theorem]{Proposition}
\newtheorem{corollary}[theorem]{Corollary}
\newtheorem{lemma}[theorem]{Lemma}
\newtheorem{conjecture}[theorem]{Conjecture}

\theoremstyle{remark}
\newtheorem{remark}{Remark}

\newtheorem*{example}{Example}

\theoremstyle{definition}
\newtheorem{definition}{Definition}
\newtheorem*{problem}{Problem}

\newcommand{\DeltaC}[2]{\Delta_{#1}(#2)}

\newcommand{\bfk}{\mathbf k}
\newcommand{\bfl}{\boldsymbol{\ell}}
\newcommand{\bftau}{\boldsymbol{\tau}}
\newcommand{\bft}{\mathbf t}
\newcommand{\teven}{\bft_{\!\textsf{e}}}
\newcommand{\todd}{\bft_{\!\textsf{o}}}

\newcommand{\compPrefactor}{\mathfrak g}
\newcommand{\comp}{\mathfrak C}

\newcommand{\tO}{\mathtt 0}                            
\newcommand{\tL}{\mathtt 1}                            

\newcommand{\digitsum}{\mathsf s}

\DeclareMathOperator{\e}{\mathrm{e}}                   
\DeclareMathOperator{\realpart}{\mathrm{Re}}
\DeclareMathOperator{\imagpart}{\mathrm{Im}}                   
\DeclareMathOperator{\LandauO}{\mathcal O}             

\newcommand{\gammaINF}{\mathfrak b}
\newcommand{\rsum}{\mathfrak B}

\title{Decomposing the sum-of-digits correlation measure}

\author{
\begin{tabular}{cc}
\begin{tabular}{c}
Bartosz Sobolewski\\
Leoben and Krak\'ow
\end{tabular}
&
\begin{tabular}{c}Lukas Spiegelhofer\\
Leoben
\end{tabular}
\end{tabular}
}
\date{}
\pgfplotsset{compat=1.18}
\begin{document}
\maketitle
\begin{abstract}
Let $\mathsf s(n)$ denote the number of ones in the binary expansion of the nonnegative integer $n$.
How does $\mathsf s$ behave under addition of a constant $t$?
In order to study the differences
\[\mathsf s(n+t)-\mathsf s(n),\]
for all $n\ge0$, we consider the associated characteristic function $\gamma_t$.
Our main theorem is a structural result on the decomposition of $\gamma_t$ into a sum of \emph{components}.
We also study in detail the case that $t$ contains at most two blocks of consecutive $\mathtt 1$s.
The results in this paper are motivated by  \emph{Cusick's conjecture} on the sum-of-digits function.
This conjecture is concerned with the \emph{central tendency} of the corresponding probability distributions,
 and is still unsolved.
\end{abstract}

\renewcommand{\thefootnote}{\fnsymbol{footnote}}
\footnotetext{
The authors were supported by the FWF project P36137-N.\\
Lukas Spiegelhofer acknowledges support by the FWF--ANR projects ArithRand (grant numbers I4945-N and ANR-20-CE91-0006),
and SymDynAr (I6750 and ANR-23-CE40-0024)
}
\renewcommand{\thefootnote}{\arabic{footnote}}

\maketitle

\smallskip\noindent

\section{Introduction}

The fundamental object we consider in this paper is the \emph{binary sum of digits} $\digitsum(n)$ of a nonnegative integer $n$, which is really just the sum of the binary digits of $n$: if $\varepsilon_0,\ldots,\varepsilon_{\nu-1}\in\{0,1\}$, then
\begin{equation*}
\digitsum\Biggl(\sum_{0\leq j<\nu}\varepsilon_j2^j\Biggr)=
\sum_{0\leq j<\nu}\varepsilon_j.
\end{equation*}
The value $\digitsum(n)$ is also given by
\begin{enumerate}

    \item the minimal number of powers of two needed to write $n$ as their sum;
    \item the expression $n-\nu_2(n!)$, where $\nu_2$ is the $2$-adic valuation function.
\end{enumerate}
The first equivalence can be seen by noting that the binary expansion of $n$ is the lexicographically largest tuple $(\delta_0,\ldots,\delta_{\nu-1})$ of nonnegative integers such that
$n=\sum_{0\leq j<\nu}\delta_j2^j$.
The equivalence to the second description goes back to Legendre, who stated the equivalent formulation
\[\nu_2(n!)=\sum_{i\ge1}\biggl\lfloor \frac n{2^i}\biggr\rfloor,\]
also valid for arbitrary primes in place of $2$ (``Legendre's formula'').
Applying Legendre's formula thrice, we obtain the important identity
\begin{equation*}
\digitsum(n)-\digitsum(n+t)=\nu_2\biggl(\!\binom{n+t}t\!\biggr)-\digitsum(t)
\end{equation*}
for integers $n,t\ge0$.
This intimate connection between
\begin{itemize}
\item[(a)] the \emph{$2$-valuation of binomial coefficients} and
\item[(b)] \emph{correlations of the binary sum-of-digits function}
\end{itemize}
is a major motivation for the study of the topics set out in the present paper.
Related to this, we wish to point out
the well-known identity
\begin{equation*}
\sum_{n\ge0}\binom{n+t}{t}x^n
=
\biggl(\frac1{1-x}\biggr)^{t+1}
\quad (t\ge0),
\end{equation*}
which provides a link to the study of \emph{powers of formal power series} (see Drmota~\cite{Drmota1994}, and Ulas~\cite{Ulas2019}).

In the article~\cite{EmmePrikhodko2017}, a probability measure on $\mathbb Z$ is defined starting from the differences $\digitsum(n+t)-\digitsum(n)$, see also~\cite{MorgenbesserSpiegelhofer2012}.
\begin{definition}
Assume that $t\ge0$ and $j$ are integers.
Define a probability measure on $\mathbb Z$ by
\begin{equation*}
\mu_t(j)\eqdef\lim_{N\to\infty}\frac1N\#\bigl\{
0\leq n<N: \digitsum(n+t)-\digitsum(n)=j
\bigr\},
\end{equation*}
and the characteristic function
\begin{equation*}
\gamma_t(\vartheta)\eqdef\sum_{j\in\mathbb Z}\mu_t(j)\exp(ij\vartheta)
=\sum_{n\ge0}\exp\bigl(i\vartheta(\digitsum(n+t)-\digitsum(n))\bigr).
\end{equation*}
\end{definition}
By B\'esineau~\cite{Besineau1972}, the limits appearing in this definition actually exist.

\medskip\noindent
\emph{Cusick's conjecture} is a statement on the structure of the probability distribution defined by $\mu_t$.
T.~W.~Cusick originally posed this as a question in 2011, and upgraded it to ``conjecture'' later.
\begin{conjecture}\label{conj_cusick}
For all $t\ge0$ we have
\begin{equation*}
c_t>\frac12,\quad\mbox{where}\quad
c_t\eqdef\sum_{j\ge0}\mu_t(j).
\end{equation*}
\end{conjecture}
We note that the expected value of $\mu_t$ equals $0$ (related to the fact that $\digitsum(n)\ll \log n$).
Therefore we feel that an interpretation of the kind ``this conjecture is plausible since larger numbers typically have a larger sum of digits'' is improper.
To us, the apparent truth of Conjecture~\ref{conj_cusick} (resp. the analogous statement $\sum_{j\le0}\mu_t(j)\ge1/2$~\cite{DrmotaKauersSpiegelhofer2016}) is still a mystery.
We did not find a convincing heuristics as to why the median should invariably be close to the expected value $0$.
This conjecture is in fact not about limits.
More precisely, for each $j\in\mathbb Z$ and $t\ge0$, the values
\begin{equation*}
v_{t,\lambda,j}\eqdef\frac1{2^\lambda}\bigl \lvert\{0\leq n<2^\lambda:\digitsum(n+t)-\digitsum(n)\geq j\}\bigr \rvert
\end{equation*}
are nonincreasing in $\lambda$,
and we have
\begin{equation*}
c_t=v_{t,\lambda,0}
\end{equation*}
for $\lambda\geq \alpha+\mu$, where $\alpha=\digitsum(t)+1$ and $2^\mu\leq t<2^{\mu+1}$~\cite[Equation~(10) and Section~3.3]{DrmotaKauersSpiegelhofer2016},~\cite{SpiegelhoferWallner2019}.

The case that the binary expansion of $t$ contains sufficiently many blocks of $\tL$s was settled by Wallner and the second author~\cite{SpiegelhoferWallner2023}.
\begin{introtheorem}{A} \label{thm_SW23}
    There exists an effective constant $N_0$ with the following property.
    If the natural number $t$ has at least $N_0$ maximal blocks of $\tL$s in its binary expansion, then $c_t > 1/2$.
\end{introtheorem}
In other words, for $t\in\mathbb N$ having the binary expansion
\begin{equation}\label{eqn_t_binary}
(t)_2=\tL^{k_{N-1}}\tO^{\ell_{N-2}}\tL^{k_{N-2}}\cdots\tO^{\ell_0}\tL^{k_0},\quad\mbox{where}\quad\begin{cases} k_1,\ldots,k_{N-1}\ge1,\\
\ell_0,\ldots,\ell_{N-2}\ge1,
\end{cases}
\end{equation}
and $N\ge N_0$,
we have $c_t>1/2$.
The proof uses a representation by means of \emph{cumulants} $\kappa_j(t)$,
\begin{equation}\label{eqn_cumulant_representation}
\gamma_t(\vartheta)=\exp\Biggl(\sum_{2\leq j\leq J}\frac{\kappa_j(t)}{j!}(i\vartheta)^j\Biggr)+\mbox{Error},
\end{equation}
near $\vartheta=0$.
An admissible value for $N_0$ could easily be found by following the proof in~\cite{SpiegelhoferWallner2023}, as our arguments do not involve ``non-constructive elements''.
The resulting figure (certainly well above $10^3$) can certainly be somewhat lowered by closer investigation, but still a \emph{brute force attack}
on the remaining cases --- few blocks of $\tL$s in binary ---
appears infeasible.
These were called the \emph{hard cases} by Cusick (2020); here the cumulant approximation~\eqref{eqn_cumulant_representation} is not good enough.

We are certain that a satisfactory proof of Cusick's conjecture necessitates new structural results.
In the present paper, such a theorem is proved, which may lead to future progress on the full solution of the conjecture.

\section{Preliminaries and notation} \label{sec_preliminaries}
We will use the notation 
$\e(x)=\exp(ix)$ for real $x$ throughout.
In this paper, $0\in\mathbb N$.
We will use Landau notation, employing the symbols $\LandauO$ and $o$.

It can be shown \cite{DrmotaKauersSpiegelhofer2016} that
\begin{equation*}
\begin{aligned}
\mu_{2t}(j)&=\mu_t(j),\\
\mu_{2t+1}(j)&=\frac12\mu_t(j-1)+\frac12\mu_{t+1}(j+1),
\end{aligned}
\end{equation*}
for $t\ge0$ and $j\in\mathbb Z$.
Defining for brevity
\begin{equation*}
\alpha\eqdef\frac{\e(\vartheta)}2,\quad\beta\eqdef\frac{\e(-\vartheta)}2,
\end{equation*}
we thus get 
\begin{equation}\label{eqn_gamma_recurrence}
\begin{aligned}
\gamma_{2t}(\vartheta)&=\gamma_{t}(\vartheta),\\
\gamma_{2t+1}(\vartheta)&=\alpha\gamma_{t}(\vartheta)+\beta\gamma_{t+1}(\vartheta),
\end{aligned}
\end{equation}
with initial condition $\gamma_0(\vartheta)=1$. In particular, plugging $t=0$ into the second equality leads to
\begin{equation*}
\gamma_1(\vartheta)=\frac{\alpha}{1-\beta}.
\end{equation*}
Due to the first equality, we may assume without loss of generality that $t$ is odd, and will usually do so in the remainder of the paper.

Using the transition matrices
\begin{equation*}
A_0=\left(\begin{matrix}1&0\\\alpha&\beta\end{matrix}\right),\quad
A_1=\left(\begin{matrix}\alpha&\beta\\0&1\end{matrix}\right),
\end{equation*}
we we may rewrite the relations \eqref{eqn_gamma_recurrence} in matrix form:
\begin{equation} \label{eqn_matrix_recurrence}
    \begin{pmatrix} \gamma_{2t} \\ \gamma_{2t+1} \end{pmatrix} = A_0
   \begin{pmatrix} \gamma_{t} \\ \gamma_{t+1} \end{pmatrix}, \qquad \begin{pmatrix} \gamma_{2t+1} \\ \gamma_{2t+2} \end{pmatrix} = A_1
   \begin{pmatrix} \gamma_{t} \\ \gamma_{t+1} \end{pmatrix}.
\end{equation} 
Let the binary expansion of $t\ge0$ be given by~\eqref{eqn_t_binary}, where $N\ge1$.
The values $k_i$ and $\ell_j$ are uniquely determined in this case
and by repeated application of \eqref{eqn_matrix_recurrence} we get

\begin{equation}\label{eqn_gamma_blocks}
\gamma_t=\left(\begin{matrix}1&0\end{matrix}\right)
A_1^{k_0} A_0^{\ell_0} \cdots A_0^{\ell_{N-2}} A_1^{k_{N-1}}
\left(\begin{matrix}1\\\gamma_1\end{matrix}\right)
\end{equation}
(compare~\cite[(1)]{MorgenbesserSpiegelhofer2012}).
This is also valid for $N=0$, in which case only the product of the outermost matrices remains, yielding $\gamma_0=1$.

Products of the matrices $A(\varepsilon)$ therefore encode all the information about the characteristic function $\gamma_t$.
Their powers are readily computed:
\begin{equation*}
A_0^\ell=\left(\begin{matrix}
1&0\\\alpha\frac{1-\beta^\ell}{1-\beta}&\beta^\ell
\end{matrix}\right), \qquad  A_1^k=\left(\begin{matrix}
\alpha^k&\beta\frac{1-\alpha^k}{1-\alpha}\\0&1
\end{matrix}\right).
\end{equation*}

Note that the factors in \eqref{eqn_gamma_blocks} appear in reverse order compared to the corresponding blocks of digits in the binary expansion of $t$.

\subsection{An extension to infinite blocks}
As we will see later, certain interesting properties of the distributions $\gamma_t$ can be discovered when prolonging a given block of $\tL$s or $\tO$s in the binary expansion of $t$ indefinitely. In order to treat these limiting cases and the functions $\gamma_t$ themselves in a unified manner, we extend the definition \eqref{eqn_gamma_blocks} by allowing ``infinite blocks'' of zeros and ones in the binary expansion of $t$. Furthermore, we may also consider $\alpha, \beta$ as independent variables (with ``default'' values $\alpha=\e(\vartheta)/2, \beta=\e(-\vartheta)/2$).
In this extended setting, more general and simpler statements can be given, while their restrictions to $\gamma_t$ for $t \in \N$ are easily recovered.


Let us make this more precise. Consider the set of extended positive integers
$\NI \eqdef \N_+ \cup \{+\infty\}$,
and their $2N$-tuples 
\[\bft = (k_0, \ell_0, \dots, k_{N-1}, \ell_{N-1}) \in  \NI^{2N}.\]
We call each such $\bft$ an \emph{extended binary expansion}. For any odd positive integer $t$ having binary expansion $\tL^{k_{N-1}} \tO^{\ell_{N-2}} \tL^{k_{N-2}} \cdots \tO^{\ell_{0}} \tL^{k_{0}}$ we can (uniquely) identify it with an extended binary expansion via the mapping
$$t  \mapsto \bft = (k_0, \ell_0, \dots, k_{N-2}, \ell_{N-2},k_{N-1}, \infty).  $$
Here, the last component being $\infty$ effectively corresponds to prepending the binary expansion of $t$ with infinitely many leading zeros. Although this may seem artificial, for the purpose of our results it will be convenient to have the same amount of blocks of $\tO$s and $\tL$s. Observe that the order of block lengths $k_i,\ell_j$ is reversed, which aligns with the matrix products in \eqref{eqn_gamma_blocks}.

\begin{example} \label{ex:extended_binary_expansion}
   $\mathbf{t} = (2,3,5,\infty)$
   is an extended binary expansion for $t = [\tL^5 \tO^3 \tL^2]_2 = 995$. 
\end{example}

It is important to note that if the length $\ell_{N-1}$ of the block of leading $\tO$s varies, then the obtained extended binary expansions are treated as different. In particular, we cannot identify $\bft$ such that  $\ell_{N-1} < \infty$ with integers. Hence, $\mathbf{t}$ from example \ref{ex:extended_binary_expansion} modified to
$ (2,3,5,2)$
does \emph{not} represent $995$.

\begin{remark}\label{rem_zero}
Throughout the remainder of the paper we use the notations $t$ as well as $\bft$.
The notation $t$ always signifies a nonnegative integer, while $\bft$ is an extended binary expansion which may (or may not) represent an integer.
\end{remark}

We move on to extend the definition of $\gamma_t$ to extended binary expansions. In a later part of this section we verify that this procedure is indeed correct. 
We first define the ``matrix powers'' $A_0^\infty$ and $A_1^\infty$ by
\begin{equation}\label{eqn_powers_limit}
A_0^\infty
\eqdef\left(\begin{matrix}
1&0\\ \frac{\alpha}{1-\beta}&0
\end{matrix}\right), \qquad  A_1^\infty\eqdef
\left(\begin{matrix}
0&\frac{\beta}{1-\alpha}\\0&1
\end{matrix}\right),
\end{equation}
where $\alpha, \beta$ are treated as variables.

Taking a hint from~\eqref{eqn_gamma_blocks},
we define
\begin{equation}\label{eqn_Mt_def}
M(\bft)\eqdef
A_1^{k_0} A_0^{\ell_0} \cdots A_1^{k_{N-1}} A_0^{\ell_{N-1}},
\end{equation}
and, reusing notation, $\gamma_\bft(\alpha,\beta)$ as the top left entry of $M(\bft)$:
\begin{equation} \label{eqn_gamma_blocks_infinite}
\gamma_{\bft}(\alpha,\beta) \eqdef \begin{pmatrix} 1 & 0 \end{pmatrix}
M(\bft)
\begin{pmatrix} 1 \\ 0 \end{pmatrix}.
\end{equation}

Importantly, for $\alpha=\e(\vartheta)/2, \beta=\e(-\vartheta)/2$ we have 
\begin{equation}\label{eqn_infinite_final_block}
A_0^{\infty} \begin{pmatrix} 1 \\ 0 \end{pmatrix} = \begin{pmatrix}
    1\\ \frac{\alpha}{1-\beta}
\end{pmatrix} =
\begin{pmatrix} 1 \\ \gamma_1 \end{pmatrix}.
\end{equation}
Hence, when $\ell_{N-1}=\infty$ and all the other components of $\bft$ are finite, we can recover  the old definition \eqref{eqn_gamma_blocks} for an integer $t$ with binary expansion $(t)_2 = \tL^{k_{N-1}} \tO^{\ell_{N-2}} \cdots \tO^{\ell_0} \tL^{k_0}$. More precisely, we then get
$$ \gamma_t(\vartheta) = \gamma_\bft\left(\frac{\e(\vartheta)}{2}, \frac{\e(-\vartheta)}{2} \right), \quad 
 \mathbf t=(k_0,\ell_0,\ldots,k_{N-2},\ell_{N-2},k_{N-1},\infty)..$$

In a similar fashion, we may define Cusick densities $c_\bft$, by considering the expansion of $\gamma_\bft$ into a bivariate power series in $\alpha, \beta$, and letting
$$  c_{\bft} =  \sum_{i \ge j }\frac1{2^{i+j}}
\bigl[\alpha^i\beta^j\bigr]
\gamma_{\mathbf t}\bigl(\alpha,\beta\bigr).$$
Again, if we let $\bft$ represent an odd integer $t$ and put $\alpha=\e(\vartheta)/2, \beta=\e(-\vartheta)/2$, we obtain the initial definition of $c_t$.

The goal of the remainder of this section is to verify that $\gamma_\bft, c_\bft$ are indeed well-defined and have desired formal and analytical properties. In particular, we would like to prove that passing to the limit $\bft^{(n)} \to \bft \in \NI^{2N}$ yields $\gamma_{\bft^{(n)}} \to \gamma_\bft$ uniformly, as well as $c_{\bft^{(n)}} \to c_\bft$. This in turn would mean that for positive integers $t$ with long blocks of $\tO$s and $\tL$s in their binary expansions, $\gamma_t, c_t$ can be approximated by setting those block lengths to $\infty$ (effectively reducing the number of parameters $k_j, \ell_j$).

\subsection{Properties of the extension} 
In the following, we consider all expressions as formal power series in the variables $\alpha$ and $\beta$.
Thus, the matrices $A_0$ and $A_1$ are $2\times 2$ matrices over the ring $\mathbb C[[\alpha,\beta]]$ of formal power series in two variables $\alpha$, $\beta$.

We now discuss formal and analytical correctness of the more general definition of $\gamma_{\bft}$. First of all, note that the sequences of formal power series given by
\[
p_m=
\alpha\frac{1-\beta^{m+1}}{1-\beta}
=\sum_{0\leq j\leq m}\alpha\beta^j\quad\mbox{and}\quad
q_m=\beta^m
\]
converge, as formal power series, to the power series $\alpha(1-\beta)^{-1}$ and $0$, respectively.
(Each coefficient is ultimately equal to the corresponding coefficient of the prospective limit.)
Of course, the same is true for $\alpha$ and $\beta$ exchanged.
By Cauchy multiplication, this immediately implies that an analogous statement is true for products of the matrices $A_1^k$ and $A_0^\ell$.
More precisely, we introduce a metric $d$ on the set $\NI^M$ of $M$-tuples of extended natural numbers, defined for $\bft=(t_0,\ldots,t_{M-1}), \bft'=(t'_0,\ldots,t'_{M-1})$ by
\begin{equation*}
\begin{aligned}
d(\bft,\bft')&\eqdef \max_{0\leq j<M}
\bigl\lvert f(t_j)-f(t'_j)\bigr\rvert,
\quad\mbox{where}\quad
f(n)\eqdef \frac1{n+1}
\mbox{ and }
f(\infty)\eqdef 0.
\end{aligned}
\end{equation*}
If $\bft^{(n)}\to \bft$ with respect to this metric (i.e. component-wise),
the sequence of matrices
$M\bigl(\bft^{(n)}\bigr)$
of power series
converges entry-wise to a limit independent of the chosen sequence $(\bft^{(n)})_n$.

In order to translate this continuity statement to an ``analytical'' one, we state the following lemma.
\begin{lemma}\label{lem_LR_splitting}
Assume that $\bft=(k_0,\ell_0,\ldots,k_{N-1},\ell_{N-1})\in\NI^{2N}$.
Let the formal power series $\mathfrak a$, $\mathfrak b$, $\mathfrak c$, and $\mathfrak d$ in $(\alpha,\beta)$ be chosen such that
\[
M(\bft)
=
\left(\begin{matrix}\mathfrak a(\alpha,\beta)&\mathfrak b(\alpha,\beta)\\\mathfrak c(\alpha,\beta)&\mathfrak d(\alpha,\beta)\end{matrix}\right).\]
Then
\begin{equation}\label{eqn_sum_of_coefficients}
\begin{aligned}
\sum_{i,j\ge0}2^{-(i+j)}\bigl[\alpha^i\beta^j\bigr]\bigl(\mathfrak a(\alpha,\beta)+\mathfrak b(\alpha,\beta)\bigr)=1,\\
\sum_{i,j\ge0}2^{-(i+j)}\bigl[\alpha^i\beta^j\bigr]\bigl(\mathfrak c(\alpha,\beta)+\mathfrak d(\alpha,\beta)\bigr)=1.
\end{aligned}
\end{equation}
Also, the coefficients of the formal power series $\mathfrak a,\mathfrak b,\mathfrak c,\mathfrak d$ are nonnegative.
\end{lemma}
\begin{proof}
Nonnegativity immediately follows by inspecting the matrices $A_0$ and $A_1$, as the power series $(1-\alpha)^{-1}$ and $(1-\beta)^{-1}$ satisfy this property.
The first statement is proved by an easy induction.
Its core is the fact that the properties~\eqref{eqn_sum_of_coefficients} are preserved under multiplication by $A_0$, $A_1$, $A_0^\infty$, or $A_1^\infty$ from the right. 
We skip the details.
\end{proof}

\subsubsection{Evaluating the formal power series $\gamma$} \label{sec_evaluating}

Assume that $\alpha$ and $\beta$ are complex numbers of absolute values $\lvert\alpha\rvert, \lvert\beta\rvert\leq1/2$.
We may evaluate $(1-\alpha)^{-1}$ and $(1-\beta)^{-1}$, and thus assign a complex value to $\gamma_{\bft}(\alpha,\beta)$.
Now assume that $\bft^{(n)}\to \bft$, which gives rise to a convergent sequence of formal power series in $(\alpha,\beta)$:
\begin{equation}\label{eqn_formal_convergence}
\gamma_{\bft^{(n)}}(\alpha,\beta)
\to\gamma_{\bft}(\alpha,\beta)\quad\mbox{as}\quad n\to\infty.
\end{equation}
Hence, for each $M\ge0$ and $n\ge n_0(M)$, we have
\begin{equation}\label{eqn_limit_swap_0}
2^{-(i+j)}\bigl[\alpha^i\beta^j\bigr]\gamma_{\mathbf t^{(n)}}(\alpha,\beta)=
2^{-(i+j)}\bigl[\alpha^i\beta^j\bigr]\gamma_{\bft}(\alpha,\beta)
\end{equation}
for all $i,j\leq M$.
By Lemma~\ref{lem_LR_splitting},
\[
\sum_{\max(i,j)>M}\textrm{LHS}=o(1)\quad\mbox{and}\quad
\sum_{\max(i,j)>M}\textrm{RHS}=o(1),\]
as $M\to\infty$.
Consequently, we have \emph{uniform convergence}
\begin{equation}\label{eqn_uniform_convergence}
\gamma_{\mathbf t^{(n)}}\to
\gamma_{\mathbf t}
\end{equation}
of the corresponding sequence of functions $B^2\to\mathbb C$, where $B=\{z\in\mathbb C:\lvert z\rvert\leq 1/2\}$.
Also, we have the uniform bound
\begin{equation*}
\bigl\lvert\gamma_{\mathbf t}(\alpha,\beta)\bigr\rvert\leq1
\quad\mbox{for all}\quad
\alpha,\beta\in B.
\end{equation*}

\noindent\textbf{Recovering the characteristic function.}
We already noted before that the values $\gamma_t(\vartheta)$ are contained as a special case.
To see this, assume that $t$ has a binary expansion as in~\eqref{eqn_t_binary}.
By~\eqref{eqn_infinite_final_block},
the substitution $\alpha=\e(\vartheta)/2$, $\beta=\e(-\vartheta)/2$
leads to
\begin{equation}\label{eqn_gamma_generalization}
\gamma_t(\vartheta) = \gamma_{\mathbf t}\biggl(\frac{\e(\vartheta)}2,\frac{\e(-\vartheta)}2\biggr),
\quad\mbox{where}\quad
\mathbf t=(k_0,\ell_0,\ldots,k_{N-2},\ell_{N-2},k_{N-1},\infty).
\end{equation}
 Lemma~\eqref{lem_LR_splitting} nicely illustrates the fact that the right hand side does in general does not define a characteristic function, as $\mathfrak b$ may be nonzero.
However, in the case $\ell_{N-1}=\infty$, the right column of the product of matrices is in fact identically $0$, which aligns with the fact that $\gamma_t$ is indeed a characteristic function.

\subsubsection{Generalized densities for Cusick's conjecture}
In this vein, the bivariate power series $\gamma_{\mathbf t}(\bl,\bl)$, defined for any $\bft\in\NI^{2N}$, gives rise to a Fourier series with coefficients
\begin{equation}\label{eqn_deltac_def}
\mu_{\mathbf t}(m)
\eqdef
\sum_{i-j=m}
\frac1{2^{i+j}}
\bigl[\alpha^i\beta^j\bigr]
\gamma_{\mathbf t}\bigl(\alpha,\beta\bigr)
,
\end{equation}
viz.
\begin{equation*}
\gamma_{\mathbf t}\left(\frac{\e(\vartheta)}{2}, \frac{\e(-\vartheta)}{2}  \right) = \sum_{m \in \Z} \mu_{\mathbf t}(m)\e(m \vartheta).
\end{equation*}

Taking~\eqref{eqn_gamma_generalization},~\eqref{eqn_deltac_def} into account, we see that
\begin{equation*}
\mu_t(m)=\mu_{\bft}(m)\quad\mbox{for}\quad
\mathbf t=(k_0,\ell_0,\ldots,k_{N-2},\ell_{N-2},k_{N-1},\infty),
\end{equation*}
where the binary expansion of $t$ is given by $k_j$ and $\ell_j$ as before.
Note that $\mu_{\bft}$ is defined for general $2N$-tuples $\bft$.
In particular, tuples $\bft\in\N^{2N}$ are admissible, but leading zeros are important in our setting:
the families $\mu_{\bft}$ and $\mu_t$ may be different if the leading block of zeros is not infinite (the left hand side always sums to $1$, while this is not true in general for the right hand side, see Lemma~\ref{lem_LR_splitting}).

We have $\mu_{\mathbf t}(m)\ge0$, and by Lemma~\ref{lem_LR_splitting}, $\sum_{m\in\mathbb Z}\mu_{\mathbf t}(m)\le1$.
We may therefore define values $\Delta_{\mathbf t}(a)$ by means of the absolutely convergent series
\begin{equation}\label{eqn_Delta_def}
\begin{aligned}
\Delta_{\mathbf t}(a)&\eqdef \sum_{m\geq a} \mu_{\mathbf t}(m)
=\sum_{i-j\ge a}\frac1{2^{i+j}}
\bigl[\alpha^i\beta^j\bigr]
\gamma_{\mathbf t}\bigl(\alpha,\beta\bigr),
\quad\mbox{in particular,}\\
\quad
c_{\mathbf t}&\eqdef 
\Delta_{\mathbf t}(a)
=\sum_{m \geq 0} \mu_{\mathbf t}(m).
\end{aligned}
\end{equation}
In particular, if $\bft$ is the extended binary expansion of $t$, we have
\begin{equation*}
c_t = c_{\bft}.
\end{equation*}

The following proposition shows that computing $c_{\mathbf t}$ commutes with limits.
\begin{proposition} \label{prop_limit_commute}
Let $\mathbf t\in\NI^{2N}$, and $a\in\mathbb Z$.
Assume that $\mathbf t^{(n)}\in\NI^{2N}$, for $n\ge0$, are such that
$\mathbf t^{(n)}\to\mathbf t$.
Then 
\begin{equation*}
\lim_{n\to\infty}  \Delta_{\mathbf t^{(n)}}(a)= \Delta_{\mathbf t}(a).
\end{equation*}
\end{proposition}
\begin{proof}
By~\eqref{eqn_formal_convergence} and~\eqref{eqn_Delta_def}, we need to show that
\[
\sum_{i-j\ge a}\frac1{2^{i+j}}
\bigl[\alpha^i\beta^j\bigr]
\gamma_{\mathbf t}\bigl(\alpha,\beta\bigr)
=
\lim_{n\to\infty}
\sum_{i-j\ge a}\frac1{2^{i+j}}
\bigl[\alpha^i\beta^j\bigr]
\gamma_{\mathbf t^{(n)}}\bigl(\alpha,\beta\bigr).
\]
Note first that, for all $i,j\ge0$,
\[
\bigl[\alpha^i\beta^j\bigr]
\gamma_{\mathbf t}\bigl(\alpha,\beta\bigr)
=
\bigl[\alpha^i\beta^j\bigr]
\lim_{n\to\infty}
\gamma_{\mathbf t^{(n)}}\bigl(\alpha,\beta\bigr)
=\lim_{n\to\infty}
\bigl[\alpha^i\beta^j\bigr]
\gamma_{\mathbf t^{(n)}}\bigl(\alpha,\beta\bigr),
\]
as the sequence $n\mapsto\gamma_{\bft^{(n)}}$ stabilizes component-wise. 
For any $M>0$, we may interchange summation and transition to a limit for the restricted sums over
\[\bigl\{(i,j)\in\mathbb N^2:i-j\ge a, i\le M,j\le M\bigr\},\]
while the remaining sums are $o(1)$ as $M\to\infty$, by Lemma~\ref{lem_LR_splitting}.
\end{proof}

\section{Main theorems} \label{sec_main} 
The content of our main theorem is a decomposition of $\gamma_{\bft}$ into a sum of \emph{components} (defined in~\eqref{eqn_component_decomp},~\eqref{eqn_comp_def} below).
Let us consider the sum-of-digits case, where $\alpha=\e(\vartheta)/2$, $\beta=\e(-\vartheta)/2$.
The presence of the term $(1-\alpha-\beta)^n$ in~\eqref{eqn_component_decomp} below
causes the Taylor expansion of component $n$ to start with $\vartheta^{2n}$ (at the earliest).
Thus the choice of the name ``component'' becomes apparent.

In order to state the theorem, we introduce some notation.
Assume that $N\ge1$ is an integer, and $\bft=(k_0,\ell_0,\ldots,k_{N-1},\ell_{N-1})\in\NI^{2N}$.
A crucial role in our theorem is played by the sequence of lengths of $\tL$- and $\tO$-blocks.
We thus write
{\everymath={\displaystyle}
\begin{equation*}
\begin{array}{l@{\hspace{3pt}}l@{}l@{\hspace{0.5em}}l@{\hspace{3pt}}c@{\hspace{3pt}}c@{\hspace{3pt}}c@{\hspace{3pt}}c@{\hspace{3pt}}c@{\hspace{3pt}}c@{\hspace{3pt}}c@{\hspace{3pt}}c@{\hspace{3pt}}c}
\teven&\eqdef&&\bigl(&k_0,&0,&k_1,&0,&\ldots,&k_{N-1},&0&\bigr),\\[3mm]
\todd&\eqdef&&\bigl(&0,&\ell_0,&0,&\ell_1,&\ldots,&0,&\ell_{N-1}&\bigr).
\end{array}
\end{equation*}
}
Our decomposition involves a sum over all subsets of blocks, indexed by 
$\bftau\in\{\tO,\tL\}^{2N}$. Define
{\everymath={\displaystyle}
\begin{equation}
\begin{array}{r@{\hspace{1mm}}l}
\lvert\bftau\rvert_\tL&\eqdef\#\{0\leq j<2N:\bftau_j=\tL\},\\[2mm]
\lvert\bftau\rvert_{\tL\tO}&\eqdef\#\{0\leq j<2N-1:(\bftau_{j},\bftau_{j+1})=(\tL,\tO)\}.
\end{array}
\end{equation}
}
The introduction of $\lvert\bftau\rvert_{\tL\tO}$ seems artificial at this point, but in fact this quantity is a key element of the problem.
 We will comment on this later on (see Section~\ref{sec_remarks}).

For convenience, we extend the inner product ``$\cdot$''  to $\NI^{2N}$ by the standard convention $0 \times\infty=\infty \times0=0$ and $a+\infty=\infty+a=\infty+\infty=\infty$,
and we set $\alpha^\infty=\beta^\infty=0$.

\bigskip
\begin{theorem}\label{thm_components}
Let $N\ge1$ be an integer.
For all $\mathbf t\in\NI^{2N}$ we have
\begin{equation}\label{eqn_gamma_components}
\gamma_{\mathbf t}(\alpha,\beta)=
\sum_{0\leq n\leq N}
(-1)^n
\compPrefactor_{N,n}(\alpha,\beta)
\comp_{N,n}(\alpha,\beta),
\end{equation}
where
\begin{equation}\label{eqn_component_decomp}
\compPrefactor_{N,n}(\alpha,\beta)
\eqdef
(\alpha\beta)^{N-n}
\frac{(1-\alpha-\beta)^n}{((1-\alpha)(1-\beta))^N},
\end{equation}
\begin{equation}\label{eqn_comp_def}
\comp_{N,n}(\alpha,\beta)\eqdef
\sum_{\substack{\bftau\in \{\tO,\tL\}^{2N}\\
\lvert\bftau\rvert_{\tL\tO}=n}}
(-1)^{\lvert\bftau\rvert_{\tL}}
\alpha^{\bftau\cdot\teven}\beta^{\bftau\cdot\todd}.
\end{equation}
\end{theorem}
Specializing to the sum-of-digits case $\alpha=\e(\vartheta)/2$, $\beta=\e(-\vartheta)/2$, we immediately obtain the following corollary.
\begin{corollary} 
Setting 
\begin{equation*}
\begin{aligned}
\compPrefactor_{N,n}(\vartheta)\eqdef
\compPrefactor_{N,n}\biggl(\frac{\e(\vartheta)}2,\frac{\e(-\vartheta)}2\biggr)
&=
\frac{\bigl(4-4\cos\vartheta\bigr)^n}
{\hspace{2pt}(5-4\cos\vartheta)^N},\\
\comp_{N,n}(\vartheta)\eqdef
\comp_{N,n}\biggl(\frac{\e(\vartheta)}2,\frac{\e(-\vartheta)}2\biggr)&=
\sum_{\substack{\bftau\in \{\tO,\tL\}^{2N}\\\bftau\cdot\bft<\infty\\
\lvert\bftau\rvert_{\tL\tO}=n}}
(-1)^{\lvert\bftau\rvert_{\tL}}
\frac{\e\bigl(\bftau\cdot(\teven-\todd)\vartheta\bigr)}{2^{\bftau\cdot(\teven+\todd)}},
\end{aligned}
\end{equation*}
we have
\begin{equation}\label{eqn_gamma_kl_rho_alt}
\gamma_t(\vartheta)=
\sum_{0\leq n\leq N}\
(-1)^n
\compPrefactor_{N,n}(\vartheta)
\comp_{N,n}(\vartheta).
\end{equation}

\end{corollary}
From this representation we can clearly see the fact~\eqref{eqn_uniform_convergence}, proved before, that prolonging an arbitrary block of $\tO$s or $\tL$s step by step causes uniform convergence, of order $1$ and rate $1/2$, of the corresponding characteristic functions $\gamma_t(\vartheta)$.
In particular, words $\bftau$ hitting an infinite block do not contribute to the sum.

\begin{example}
Let $t=153=(10011001)_2$. We have
$\teven=(1,0,2,0,1,0)$ and $\todd=(0,2,0,2,0,\infty)$.
Distinguishing between $32$ cases (with the help of Sage~\cite{sage}) we obtain
\begin{equation*}
\begin{aligned}
\comp_{3,0}(\vartheta)&=1,\\
\comp_{3,1}(\vartheta)&=
-\frac{33}{64}\e(-2\vartheta)
+\frac{17}{64}\e(-\vartheta)
+\frac{31}{256}
-\frac{17}{16}\e(\vartheta)
-\frac14\e(2\vartheta),\\
\comp_{3,2}(\vartheta)&=
\frac1{16}\e(-4\vartheta)
-\frac1{16}\e(-3\vartheta)
+\frac1{64}\e(-2\vartheta)
+\frac14\e(-\vartheta)
-\frac18
-\frac1{16}\e(\vartheta)
+\frac{9}{32}\e(2\vartheta)
+\frac14\e(3\vartheta)
,\\
\comp_{3,3}(\vartheta)&=-\frac1{16}\e(4\vartheta).
\end{aligned}
\end{equation*}

\begin{figure}[!ht]
\begin{tabular}{cc}
\includegraphics[width=0.45\linewidth]{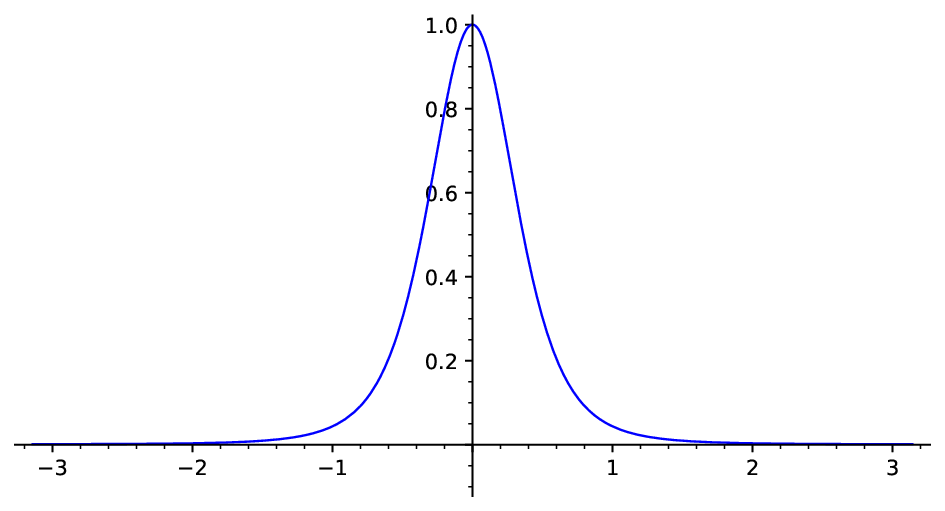}&
\includegraphics[width=0.45\linewidth]{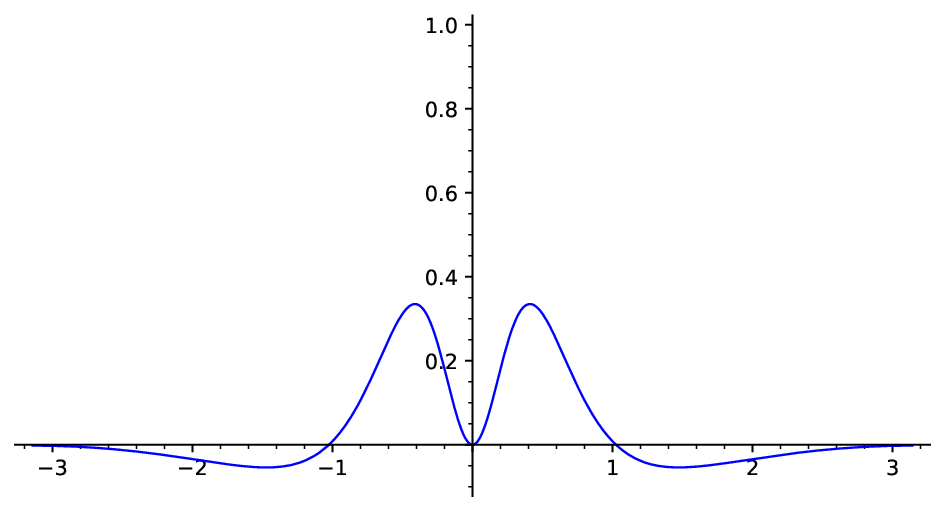}\\
\includegraphics[width=0.45\linewidth]{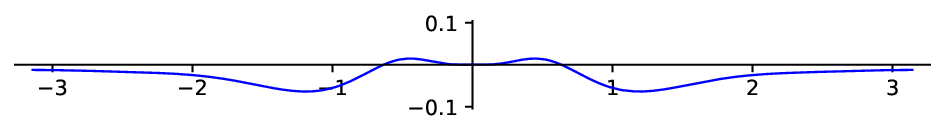}&
\includegraphics[width=0.45\linewidth]{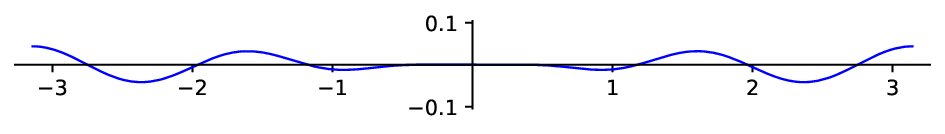}
\end{tabular}
\label{fig_components}
\caption{Real part of $(-1)^n\compPrefactor_{3,n}(\vartheta)\comp_{3,n}(\vartheta)$ for $t=153$ and $n\in\{0,1,2,3\}$.}
\end{figure}
\end{example}

\bigskip\noindent\textbf{Nonnegative coefficients.}
The representation~\eqref{eqn_comp_def} has the advantage that the components are ordered by the size of the central quantity $\lvert\bftau\rvert_{\tL\tO}$.
Meanwhile, the sign of the coefficients of $\compPrefactor_{N,n}$ is not so clear.

We state an alternative form of Theorem~\ref{thm_components} ``reversing these two parameters'': the signs of the appearing coefficients are obvious in this new formulation, but as a drawback, $\lvert\bftau\rvert_{\tL\tO}$ will appear in the summands.
To this end, let us use the abbreviation
\begin{equation*}
\gammaINF_m(\alpha,\beta) \eqdef
\compPrefactor_{m,0}(\alpha,\beta)
=\left(\frac{\beta\alpha}{(1-\alpha)(1-\beta)}\right)^m,
\end{equation*}
which is a bivariate power series in $\alpha$ and $\beta$ having nonnegative coefficients.
We express $\compPrefactor_{N,n}$ as a linear combination of the quantities $\compPrefactor_{m,0}$, where $N-n\leq m\leq N$,
by expanding the binomial expansion of $(1-\alpha-\beta)^n = ((1-\alpha)(1-\beta) -\alpha\beta)^n$.
We obtain
\[\compPrefactor_{N,n}=\sum_{j=0}^n (-1)^{n-j} \gammaINF_{N-j}\]
by
\eqref{eqn_component_decomp}, and thus the following proposition holds true.

\begin{proposition}\label{prp_components}
Let $N\ge1$ be an integer.
For all $\bft \in \NI^{2N}$ we have
\begin{equation}\label{eqn_gamma_components_alt}
\gamma_{\bft}(\alpha,\beta)=
\sum_{\bftau\in \{\tO,\tL\}^{2N}}
(-1)^{\lvert\bftau\rvert_{\tL}} \alpha^{\bftau\cdot\teven}\beta^{\bftau\cdot\todd}
 \sum_{j=0}^{\lvert\bftau\rvert_{\tL\tO}} \binom{\lvert\bftau\rvert_{\tL\tO}}{j} (-1)^j \gammaINF_{N-j}(\alpha,\beta).
\end{equation}
\end{proposition}
\begin{remark}
The main results above may be translated to statements on $c_{\bft}$, by substituting $\alpha=\e(\vartheta)/2$ and $\beta=\e(-\vartheta)/2$, and expanding the components.
This will be made explicit in Section~\ref{sec:simplest_cases}.
\end{remark}

\subsection{Remarks and corollaries}\label{sec_remarks}
Several remarks and observations on the main theorem are in order.
\begin{enumerate}
\item (Graduation).
The decomposition~\eqref{eqn_gamma_components} of $\gamma_\bft$ into components satisfies the relation
\begin{equation*}
\bigl(\mathfrak g_{N,n}\comp_{N,n}\bigr)(\vartheta)=\LandauO\bigr(\vartheta^{2n}\bigr)\quad\mbox{as}\quad\vartheta\rightarrow0.
\end{equation*}
\item (Base case).
Consider the (sum-of-digits) case that $\ell_{N-1}=\infty$.
If $\lvert \bftau\rvert_{\tL\tO}=0$, we have $\bftau=\tO^a\tL^b$ for some nonnegative integers $a,b$ such that $a+b=2N$.
The only contributing $\bftau$ with this property satisfies $b=0$ (since selecting an infinite block causes the contribution to be $0$), that is, $\bftau=(\tO,\ldots,\tO)$.
We therefore have
\begin{equation*}
\bigl(\mathfrak g_{N,0}\comp_{N,0}\bigr)(\vartheta)=
\mathfrak g_{N,0}(\vartheta)
=\gammaINF_N(\vartheta)
=\biggl(\frac{1}{5-4\cos\vartheta}\biggr)^N.
\end{equation*}

This ``base case'' always appears, but has to be completed by further components as soon as blocks of finite length are present.
If $\bft = (\infty,\infty,\ldots,\infty) \in \NI^{2N}$,
no other cases occur, and by Theorem \ref{thm_components} we have
$\gamma_\bft=\gammaINF_N$.

\item (Counting factors $\tL\tO$). Besides the number $\lvert\bftau\rvert_{\tL}$ of selected summands $k_j,\ell_j$, also the factor-counting function $\lvert\bftau\rvert_{\tL\tO}$ appears,
which counts occurrences of $\tL\tO$ at both even and odd indices.
That is, equations~\eqref{eqn_gamma_components} and~\eqref{eqn_gamma_components_alt} feature
the number of times that either a block of ones is selected while block of zeros next to it (to the left) is ignored, or a block of zeros is selected while the (left) adjacent block of ones is ignored.
Via the presence of the term $\lvert\bftau\rvert_{\tL\tO}$, the ordering of blocks enters into Equations~\eqref{eqn_gamma_components} and~\eqref{eqn_gamma_components_alt}.
In fact, the arrangement of blocks is \emph{only} detected by virtue of the term $\lvert\bftau\rvert_{\tL\tO}$,
which contributes strongly to our impression that this term is the ``right'' quantity to consider. 

Counting $\tL\tO$-blocks thus appears to be an integral component 
of the study of the correlation $\digitsum(n+t)-\digitsum(n)$, which came very much as a surprise to the authors.
We consider this as a manifestation of the principle --- which we might call ``nonlocality'' --- that blocks of $\tO$s or $\tL$s cannot be considered independently from each other when adding two binary integers,
caused by carries propagating through arbitrarily many adjacent blocks.
To borrow a formulation by J.~S.~Bell~\cite{Bell1980},
we consider it a merit of our theorem to bring out the dependence on this term so explicitly that it cannot be ignored.

\item (Summary).
We may therefore give the following intuitive explanation of Theorem~\ref{thm_components}.
As the \emph{length of the shortest block} tends to $\infty$, the characteristic function approaches the base case $\gammaINF_N$, since --- rougly --- the contribution of a block decreases with increasing length of the block.
As the \emph{number of blocks} increases, the corresponding sequence of base cases converges to a Gaussian, in a suitable sense, as we consider powers of a function $1+c\vartheta^2+\LandauO(\vartheta^3)$, where $c<0$.
The contributions of the non-base components impose a certain \emph{perturbation} of the Gaussian behaviour. These components are controlled by the finite blocks, whose contributions are mixed according to their order of appearance.

\end{enumerate}

\subsubsection{Dissection by infinite blocks}\label{sec_dissection}

Note that those $\bftau$ hitting an ``infinite blocks of $\tO$s or $\tL$s'' do not contribute to the sum.
This fact implies that the order in which the blocks of $\tO$s appear does not play a role if they are separated by ``infinite blocks of $\tL$s'', see Corollary~\ref{cor_swap} below.
Now and in the following, we let $\bfk\ast\bfl$ denote the \emph{interleaving} of the sequences $\bfk,\bfl\in\NI^{N}$:
\begin{equation}\label{eqn_interleaving_def}
\bigl(k_0,\ldots,k_{N-1}\bigr)\ast\bigl(\ell_0,\ldots,\ell_{N-1}\bigr)\eqdef
\bigl(k_0,\ell_0,k_1,\ell_1,\ldots,k_{N-1},\ell_{N-1}\bigr).
\end{equation}
Using this convenient notation, we can state the result as follows.
\begin{corollary}\label{cor_swap}
Let $\bfk=(k_0,\ldots,k_{N-1}),\bfl=(\ell_0,\ldots,\ell_{n-1})\in\NI^{N}$.
\begin{itemize}
\item If $\ell_0=\cdots=\ell_{N-1}=\infty$, we have
\begin{equation}\label{eqn_cor_swap_0}
\gamma_{\bfk\ast\bfl}=\gamma_{\sigma\bfk\ast\bfl}
\end{equation}
for any permutation $\sigma$ of $\{0,\ldots,{N-1}\}$, where
\[\sigma\bfk\eqdef\bigl(k_{\sigma^{-1}(0)},\ldots,k_{\sigma^{-1}(N-1)}\bigr).\]

\item Analogously, if $k_0,\ldots,k_{N-1}=\infty$, we have
\begin{equation}\label{eqn_cor_swap_1}
\gamma_{\bfk\ast\bfl}=\gamma_{\bfk\ast\sigma\bfl}
\end{equation}
for any permutation $\sigma$ of $\{0,\ldots,N-1\}$ such that $\sigma(N-1)=N-1$.
\end{itemize}
\end{corollary}
Equations~\eqref{eqn_cor_swap_0} and~\eqref{eqn_cor_swap_1} are identities of formal power series in two variables $\alpha$ and $\beta$.

\begin{proof}[Proof of Corollary~\ref{cor_swap}]
We prove the first statement.
As $\bftau$ runs through $\{\tO,\tL\}^{2N}$,
every set $\subseteq \{0,2,4,\ldots,2N-2\}$ occurs as the set of positions
of appearances of $\tL\tO$.
Since we consider ``infinite blocks of $\tO$s'' only,
vectors $\bftau$ having an occurrence of $\tL\tO$ at an odd position do not contribute to the sum.
It follows that
\[\comp_{N,n}=
\sum_{\substack{\bftau\in \{\tO,\tL\}^{2N}\\
\lvert\bftau\rvert_{\tL\tO}=n}}
(-1)^{\lvert\bftau\rvert_{\tL}}
\alpha^{\bftau\cdot\teven}\beta^{\bftau\cdot\todd}
=
(-1)^{n}
\sum_{\substack{J\subseteq\{0,\ldots,N-1\}\\\lvert J\rvert=n}}
\alpha^{\sum_{j\in J}k_j}.
\]
This expression is obviously invariant under the action of $\sigma$ on $\bfk$.

\noindent In the proof of the second statement,
we need to consider $\ell_{N-1}$ separately, since $\tL\tO$ cannot appear at position $2N-1$.
Analogously to the first case, the set of positions of appearances of $\tL\tO$ is a subset of $\{1,3,5,\ldots,2N-3\}$, and every subset actually appears.
Therefore $\comp_{N,N}=0$, and
\[\comp_{N,n}=
\sum_{\substack{\bftau\in \{\tO,\tL\}^{2N}\\
\lvert\bftau\rvert_{\tL\tO}=n}}
(-1)^{\lvert\bftau\rvert_{\tL}}
\alpha^{\bftau\cdot\teven}\beta^{\bftau\cdot\todd}
=
\bigl(1-\beta^{\ell_{N-1}}\bigr)
(-1)^n
\sum_{\substack{J\subseteq\{0,\ldots,N-2\}\\\lvert J\rvert=n}}
\beta^{\sum_{j\in J}\ell_j}.
\]
for $0\leq n<N-1$.
This expression is not changed by the action of $\sigma$ on $\bfl$, since $\sigma(N-1)=N-1$.
\end{proof}

\begin{example}
If
\[(t_m)_2=\tL^m\tO\tL^m\tO\tO\tL^m\tO\tL^m,\quad (t'_m)_2=\tL^m\tO\tL^m\tO\tL^m\tO\tO\tL^m,\]
by the corollary we have identical limits
\[\lim_{m\rightarrow\infty}c_{t_m}
=\lim_{m\rightarrow\infty}c_{t_m'},
\]
which can be evaluated to yield the value $2737/(2^63^4)$.
\end{example}
We note that in the case of only two blocks of adjacent $\tO$s, we have equality of the $c_t$-values \emph{for each $m$}.
This follows from the \emph{digit reversal property} proved by Morgenbesser and the second author~\cite{MorgenbesserSpiegelhofer2012}.
The special case $q=2$ (the binary sum-of-digits) of the cited result also follows from our main theorem, as we show now.
\subsubsection{The ``reverse order property''}\label{sec_curious}
Morgenbesser and the second author~\cite{MorgenbesserSpiegelhofer2012} proved that $\gamma_t$ is invariant under reversal of the base-$q$ digits.
It remained unclear, intuitively, why such a ``curious property'' should hold;
we give a new proof, for the binary case, shedding some light on the situation. In particular, using Theorem~\ref{thm_components}, the (almost-) invariance of $\lvert\cdot\rvert_{\tL\tO}$ under reflection explains this unexpected behaviour.

\begin{corollary} \label{cor_reversal}
Assume that $t'$ is the integer obtained from $t$ by reversing its binary expansion. Then $\gamma_t=\gamma_{t'}$, in particular, $c_t=c_{t'}$.
\end{corollary}
\begin{proof}
It is sufficient to consider tuples
\[\bft=\bigl(0,\infty,t_2,\ldots,t_{2N-2},\infty\bigr).
\]
That is, 
$\bft=\bfk\ast\bfl$, where
$\bfk=(0,k_1,\ldots,k_{N-1})$ and $\bfl=(\infty,\ell_1,\ldots,\ell_{N-2},\infty)$. 
Digit reversal corresponds to defining $t'_2=t_{2N-2},\ldots,t'_{2N-2}=t_2$.

We also define the restricted reversal of $\bftau\in\{\tO,\tL\}^{2N}$ by
$\bftau'_0=\bftau_0$, $\bftau'_j=\bftau_{2N-j}$ for $1\leq j\leq 2N-1$.
Terms in~\eqref{eqn_Delta_kl_R} such that $\bftau_1=\tL$ or $\bftau_{2N-1}=\tL$ do not contribute, and the remaining choices of $\bftau$ satisfy the central property
\[\lvert\bftau\rvert_{\tL\tO}=\lvert\bftau'\rvert_{\tL\tO}.\]
Also,
\[\bftau'\cdot\teven'=\bftau\cdot\teven,
\quad\mbox\quad
\bftau'\cdot\todd'=\bftau\cdot\todd.
\]
Thus we have found a bijection on the set $\{\tO,\tL\}^{2N}$ for which each of the summands in~\eqref{eqn_Delta_kl_R}, for the case of $\bft$, equals a summand for the case $\bft'$.
\end{proof}

\begin{example}
We have
\[\gamma_{(\mathtt{1011010011})_2}(\pi)=
\gamma_{(\mathtt{1100101101})_2}(\pi)=\frac{17}{384}.\]
\end{example}
\subsubsection{Correlations of the Thue--Morse sequence}\label{sec_Thue_Morse}
Let us consider the special values $\gamma_t\eqdef\gamma_t(\pi)$,
corresponding to $\alpha=\beta=-1/2$.
These values are correlations of the Thue--Morse sequence, as considered in the survey~\cite{Mauduit2001} by Mauduit, and the article~\cite{BaakeCoons2024} by Baake and Coons.
Let $t\ge1$ be an integer, and $\bft\in\NI^{2N}$ the corresponding sequence of lengths of binary blocks 
(where $\ell_{N-1}=\infty$, and $k_j, \ell_j<\infty$ otherwise).
By~\eqref{eqn_gamma_kl_rho_alt}, we obtain the particularly simple decomposition into components
\begin{equation}\label{eqn_gamma_kl_rho_alt2}
\gamma_t=
\biggl(\frac19\biggr)^N
\sum_{0\leq n\leq N}
(-8)^n
C_{N,n},
\end{equation}
where
\begin{equation*}
C_{N,n}\eqdef\sum_{\substack{\bftau\in \{\tO,\tL\}^{2N-1}\\
\tau_{2N-1}=0\\
\lvert\bftau\rvert_{\tL\tO}=n}}
\bigl(-1\bigr)^{\lvert\bftau\rvert_{\tL}}
\biggl(-\frac12\biggr)^{\bftau\cdot\bft}.
\end{equation*}
In the simplest case that $t=1$, we have $\bft=(1,\infty)$,
and only two summands, corresponding to $\bftau\in\{\tO\tO,\tL\tO\}$.
Equation~\eqref{eqn_gamma_kl_rho_alt} yields $\gamma_1=1/9-8/18=-1/3$, which is consistent with $\gamma_1(\vartheta)=\e(\vartheta)/(2-\e(-\vartheta))$.

The values $\gamma_t$ satisfy the recurrence~\cite{Mauduit2001}
\begin{equation*}
\gamma_0=1,\quad
\gamma_1=-1/3,\quad
\gamma_{2t}=\gamma_t,\quad
\gamma_{2t+1}=-\frac{\gamma_t+\gamma_{t+1}}2.
\end{equation*}
These values can be arranged nicely in a table (see Figure~\ref{fig_some_values}), where the first row contains the values $\gamma_1$ and $\gamma_2$, the second row the values $\gamma_2$ through $\gamma_4$, the third $\gamma_4$ through $\gamma_8$, and so on (compare Stern~\cite{Stern1858}, Lehmer~\cite{Lehmer1929}, Northshield~\cite{Northshield2010})
For greater clarity, we scale the values by the factor $-3$.
\setlength{\extrarowheight}{3pt}
\begin{figure}
\begin{equation*}
\begin{array}{lllllllllllllllll}
1&&&&&&&&&&&&&&&&1\\
1&&&&&&&&-1&&&&&&&&1\\
1&&&&0&&&&-1&&&&0&&&&1\\
1&&-\tfrac12&&0&&\tfrac12&&-1&&\tfrac12&&0&&-\tfrac12&&1\\
1&-\tfrac14&-\tfrac12&\boxed{\tfrac14}&0&-\tfrac14&\tfrac12&\tfrac14&-1&\boxed{\tfrac14}&\tfrac12&-\tfrac14&0&\tfrac14&-\tfrac12&-\tfrac14&1
\end{array}
\end{equation*}
\caption{Some values $-3\gamma_t$, where $-3\gamma_{19}$ and $-3\gamma_{25}$ are boxed.}\label{fig_some_values}
\end{figure}
\setlength{\extrarowheight}{0pt}
For example, as the system $a+b=1$; $a-b=-1/12$ of linear equations has the unique solution $(a,b)=(11/24,13/24)$, we see that the Thue--Morse sequence changes sign on a set of asymptotic density $13/24$ upon adding $19=(10011)_2$, and also when adding the integer with reversed binary expansion, $25=(11001)_2$~\cite{MorgenbesserSpiegelhofer2012}.

\section{Proof of Theorem \ref{thm_components}} \label{sec_proof}

Let $\bft = (k_0, \ell_0, \ldots, k_{N-1}, \ell_{N-1}) \in \NI^{2N}$.
We analyze the matrix product in \eqref{eqn_Mt_def}, written in the form
 \begin{equation}   \label{eq:long_matrix_product}
 \prod_{j=0}^{N-1} (A_1^{k_j} A_0^{\ell_j}) = \frac{1}{ ((1-\alpha)(1-\beta))^N} \prod_{j=0}^{N-1} (T_{\tO\tO} + \alpha^{k_j} T_{\tL\tO} + \beta^{\ell_j} T_{\tO\tL} +\alpha^{k_j} \beta^{\ell_j} T_{\tL\tL}),
 \end{equation}
 where $T_{\tO\tO}, T_{\tL\tO}, T_{\tO\tL},  T_{\tL\tL}$ are $2 \times 2$ matrices given by
\begin{alignat*}{2}  
 &T_{\tO\tO} = \begin{pmatrix}
    \alpha\beta & 0 \\ \alpha(1-\alpha) & 0
\end{pmatrix}, 
 &&T_{\tL\tO} = \begin{pmatrix}
   1-\alpha-\beta & 0 \\ 0 & 0
\end{pmatrix},  \\  
 &T_{\tO\tL} = \begin{pmatrix}
   -\alpha\beta & \beta(1-\beta) \\ -\alpha(1-\alpha) & (1-\alpha)(1-\beta)
\end{pmatrix},  \qquad   
 &&T_{\tL\tL} = \begin{pmatrix}
    \alpha\beta & -\beta(1-\beta) \\ 0 & 0
\end{pmatrix}.\end{alignat*}

 Expanding the right-hand side of  \eqref{eq:long_matrix_product}, we see that each possible product 
  $T_{u_0} \cdots T_{u_{N-1}}$
 appears (where $u_j \in \{\tO,\tL\}^2$), and the accompanying product of monomials $\alpha^{k_j}, \beta^{\ell_j}$ is encoded by the binary word 
 $u_0 \cdots u_{N-1}$. 
 More precisely, writing $\bftau = u_0 \cdots u_{N-1} \in \{\tO,\tL\}^{2N}$, we get
\begin{equation} \label{eq:matrix_product_expanded} \prod_{j=1}^N (A_1^{k_j} A_0^{\ell_j}) = \frac{1}{ ((1-\alpha)(1-\beta))^N} \sum_{\boldsymbol{\tau} \in\{\tO, \tL\}^{2N}} \alpha^{\bftau \cdot \teven} \beta^{\bftau \cdot \todd}T_{u_0} \cdots T_{u_{N-1}},
\end{equation}
where $\teven, \todd$ are defined as in Section \ref{sec_main}.

We now investigate the products $T_{u_0} \cdots T_{u_{N-1}}$ in more detail. First, we can observe that products of two matrices among $T_{\tO\tO},T_{\tO\tL},T_{\tL\tO}, T_{\tL\tL}$ give another such matrix, multiplied by $\pm\alpha\beta$ or $\pm (\alpha+\beta-1)$. All the possible products are described in Table \ref{tab:multiplication}.
\begin{table}[h!]
\centering
\begin{tabular}{c|c c c c}  
 & $T_{\tO\tO}$ & $T_{\tL\tO}$ & $T_{\tO\tL}$ & $T_{\tL\tL}$ \\ 
 \hline
 $T_{\tO\tO}$ & $\alpha\beta T_{\tO\tO}$ & $-(\alpha+\beta-1)T_{\tO\tO}$ & $\alpha\beta T_{\tO\tL}$ & $-\alpha\beta T_{\tO\tL}$ \\ 
 $T_{\tL\tO}$ & $\alpha\beta T_{\tL\tO}$ & $-(\alpha+\beta-1)T_{\tL\tO}$ & $(\alpha+\beta-1) T_{\tL\tL}$ & $-(\alpha+\beta-1) T_{\tL\tL}$ \\
 $T_{\tO\tL}$ & $-(\alpha+\beta-1)T_{\tO\tO}$ & $(\alpha+\beta-1)T_{\tO\tO}$ & $-(\alpha+\beta-1)T_{\tO\tL}$ & $\alpha \beta T_{\tO\tL}$ \\
 $T_{\tL\tL}$ & $-\alpha\beta T_{\tL\tO}$ & $\alpha\beta T_{\tL\tO}$ & $-(\alpha+\beta-1)T_{\tL\tL}$ & $\alpha \beta T_{\tL\tL}$
\end{tabular}
\caption{Multiplication table for the matrices $T_v$}
\label{tab:multiplication}
\end{table}

It is thus clear that for any $\bftau = u_0 \cdots u_{N-1}$, we can write
\begin{equation} \label{eq:matrix_product}
    T_{u_0} \cdots T_{u_{N-1}} = (-1)^{\sigma(\bftau)} (\alpha+\beta-1)^{\lambda(\bftau)} (\alpha \beta)^{N-1-\lambda(\bftau)} T_{\nu(\bftau)},
\end{equation} 
where $\sigma(\bftau) \in \{0,1\}, \lambda(\bftau) \in \{0,1,\ldots,N-1\}$, and $\nu(\bftau) \in \{\tO\tO,\tO\tL,\tL\tO,\tL\tL\}$ are uniquely determined by $\bftau$. We claim that the functions $\nu, \sigma, \lambda$ can be explicitly described in the following way:
\begin{enumerate}
    \item[(a)] $\nu(\bftau) = xy$, where $x$ and $y$ are, respectively, the first and last digit in $\bftau$; 
    \item[(b)] $\sigma(\bftau) \equiv |\bftau|_\tL - |\nu(\bftau)|_\tL \pmod{2}$;
    \item[(c)] $\lambda(\bftau) = |\bftau|_{\tL\tO} - |\nu(\bftau)|_{\tL\tO}$.
\end{enumerate}

These formulas can be proved by simple induction on $N$. Indeed, for $N=1$ we have $\nu(\bftau)= \bftau$ and there is nothing to prove, while for $N=2$ they follow by inspection of Table \ref{tab:multiplication}.

Now, assume that (a)--(c) hold for some $N \geq 1$ and put $\mathbf{v} =  \bftau u_N =  u_0 \cdots u_{N-1} u_N$, where $u_N \in \{\tO\tO,\tO\tL,\tL\tO,\tL\tL\}$.
Writing $T_{u_0} \cdots T_{u_{N-1}} T_{u_N} = (T_{u_0} \cdots T_{u_{N-1}})\cdot T_{u_N}$, we get
$$ \nu(\mathbf{v})  = \nu(\nu(\bftau)u_N) = xy, $$
where $x$ is the first digit of $\nu(\bftau)$ and $y$ is  the last digit of $u_N$ (using (a) for $N=2$). By the inductive assumption $x$ is also the first digit of $\bftau$ (and $\mathbf{v}$), and thus we get (a) for $N+1$.

Part (b) follows from
\begin{align*} \sigma(\mathbf{v}) &\equiv \sigma(\bftau) + \sigma(\nu(\bftau)u_N) \\
&\equiv |\bftau|_\tL - |\nu(\bftau)|_\tL +  |\nu(\bftau)u_N|_\tL - |\nu(\nu(\bftau)u_N)|_\tL \\
&\equiv |\mathbf{v}|_\tL -|\nu(\mathbf{v})|_\tL \pmod{2},
\end{align*}
where we have again used (b) for $N=2$ and the inductive assumption.

Finally, in part (c) we get in a similar fashion
$$
\lambda(\mathbf{v}) = \lambda(\bftau) + \lambda(\nu(\bftau)u_N)
= |\bftau|_{\tL\tO} - |\nu(\bftau)|_{\tL\tO} +  |\nu(\bftau)u_N|_{\tL\tO} - |\nu(\mathbf{v})|_{\tL\tO}.
$$
Now, if $y$ denotes the last digit of $\bftau$ (and $\nu(\bftau)$), then 
$$ |\nu(\bftau)u_N|_{\tL\tO} - |\nu(\bftau)|_{\tL\tO} = |y u_N|_{\tL\tO} = |\bftau_{N+\tL}|_{\tL\tO} - |\bftau|_{\tL\tO},$$
and thus (c) also follows.

Recall that $\gamma_{\bft}$ is the top left entry of \eqref{eq:matrix_product_expanded}, and from equality \eqref{eq:matrix_product} we can deduce that the contribution of $T_{u_0} \cdots T_{u_{N-1}}$ is precisely
$$
   (-1)^{|\bftau|_\tL}  (\alpha+\beta-1)^{|\bftau|_{\tL\tO}} (\alpha \beta)^{N-|\bftau|_{\tL\tO}},  
$$
where $\bftau = u_0 \cdots u_{N-1}$. Summing over $\bft$, and grouping according to $|\bftau|_{\tL\tO}$, we reach the desired formula for $\gamma_{\bft}$.

\section{Explicit computation of $c_\bft$} \label{sec:simplest_cases}

Using Theorem \ref{thm_components} or Proposition \ref{prp_components}, one can compute $c_\bft$ by inserting $\alpha=\e(\vartheta)/2, \beta=\e(-\vartheta)/2$ and expanding $\gamma_\bft$ into a Fourier series. The formulation in Proposition \ref{prp_components} is arguably better suited for this, since it only involves the components $\gammaINF_m$, as opposed to the $\compPrefactor_{N,n}$ depending on two parameters. Hence, we will only consider this formulation.

Let $b_j$ denote the coefficients in the mentioned Fourier expansion:
$$\gammaINF_m\left(\frac{\e(\vartheta)}{2},\frac{\e(-\vartheta)}{2}\right) =\frac{1}{(5-4\cos \vartheta)^m}=\sum_{j \in \Z} b_{m,j} \e(j\vartheta).$$
For $a \in \Z$ it is useful to have an expression for the sum of these coefficients over $j \geq a$:
$$\rsum_m(a) \eqdef \sum_{j \geq a} b_{m,j}.$$
Equivalently, we could define $\rsum_m(a)$ without substituting for $\alpha,\beta$, as follows:
\begin{equation*}
\rsum_m(a)=\sum_{i-j\ge a}
\frac1{2^{i+j}}
\bigl[\alpha^i\beta^j\bigr]
\gammaINF_m(\alpha,\beta)
=\sum_{i-j\ge a}
\frac1{2^{i+j}}
\bigl[\alpha^i\beta^j\bigr]
\biggl(\frac{\beta\alpha}{(1-\alpha)(1-\beta)}
\biggr)^m.
\end{equation*}
Then, by Proposition \ref{prp_components} we have the following formula, where for $\bft=(k_0,\ell_0,\ldots,k_{N-1},\ell_{N-1})\in\NI^{2N}$ we again use the notation $\teven = (k_0,0,\ldots,k_{N-1},0)$ and $\todd=(k_0,0,\ldots,k_{N-1},0)$.

\begin{corollary}\label{cor_c_explicit}
Let $N\ge1$ be an integer, $\bft\in \NI^{2N}$, and $a\in\mathbb Z$.
We have\begin{equation}\label{eqn_Delta_kl_R}
\begin{aligned}
c_\bft&=
\sum_{\substack{\bftau\in \{\tO,\tL\}^{2N}\\
\bftau\cdot\bft<\infty}}
(-1)^{\lvert\bftau\rvert_{\tL}} 2^{-\bftau\cdot\bft}
\mathcal D_{\lvert\bftau\rvert_{\tL\tO}}\bigl(\bftau\cdot(\todd-\teven)\bigr),
\end{aligned}
\end{equation}
where
\begin{equation*}
\mathcal D_{\ell}(r)\eqdef
\sum_{j=0}^{\ell} \binom{\ell}{j} (-1)^j
\rsum_{N-j}(r).
\end{equation*}
\end{corollary}

\begin{remark}\label{rem_canonical_reduction}
Corollary~\ref{cor_c_explicit} contains an explicit reduction of $c_{\bft}(a)$ to the canonical cases $\rsum_n(a)$.
The contribution of the ``base case'' $\bftau=(0,\ldots,0)$ 
to $c_{\bft}(a)$ is just $\rsum_N(0)>1/2$ (see Proposition~\ref{prp_dominant_component} below).
Meanwhile, the expression $\mathcal D_\ell(r)$ is a higher difference of values $\rsum_j(r)$.
For example, $\ell=\lvert \bftau\rvert_{\tL\tO}=2$ yields
\[\mathcal D_2(r)=\rsum_{N-0}(r)-2\rsum_{N-1}(r)+\rsum_{N-2}(r)=
\bigl(\rsum_{N-0}(r)-\rsum_{N-1}(r)\bigr)-\bigl(\rsum_{N-1}(r)-\rsum_{N-2}(r)\bigr).\]
These higher differences seem to be quite small in absolute value, and appear with a factor $\leq 2^{-\lvert\bftau\rvert_{\tL}}$.
A precise understanding of these objects might provide a new path towards Cusick's conjecture.
\end{remark}
For given $m$, the coefficients $b_{m,j}$ and their sums $\rsum_m(a)$ can be computed in a routine way. As an example, the five initial cases are given as follows.
\begin{align*}
   \rsum_0(a) &= \begin{cases}
0 &\text{if } a > 0, \\
1 &\text{if } a \leq 0,
       \end{cases}
   \\
   \rsum_1(a) &= \begin{cases}
 \frac{2^{-a}}{3} \cdot 2 &\text{if } a > 0, \\
1  - \frac{2^a}{3} &\text{if } a \leq 0,
       \end{cases}\\
   \rsum_2(a) &= \begin{cases}
\frac{2^{-a}}{27}(6a + 16) &\text{if } a > 0, \\
1 - \frac{2^a}{27}(-3a+11) &\text{if } a \leq 0,
       \end{cases}\\
   \rsum_3(a) &= \begin{cases}
\frac{2^{-a}}{81}  \left(3 a^2+21 a+46\right) &\text{if } a > 0, \\
1-\frac{2^a}{162}  \left(3 a^2-27 a+70\right) &\text{if } a \leq 0,
       \end{cases}\\ 
   \rsum_4(a) &= \begin{cases}
\frac{2^{-a}}{2187} \left(9 a^3+117 a^2+600 a+1216\right) &\text{if } a > 0, \\
1-\frac{2^a}{4374} \left(-9 a^3+144 a^2-861 a+1942\right) &\text{if } a \leq 0.
       \end{cases}
\end{align*}
\subsection{Evaluating the base cases}\label{sec_base_case}
We have the following lower bound for the values $\rsum_m(0)$.
\begin{proposition}\label{prp_dominant_component}
Let $m\ge1$ be an integer, and 
$\bft=(\infty,\ldots,\infty)\in\NI^{2m}$.
We have
\begin{equation*}
\rsum_m(0)=\DeltaC{\bft}{0}
>\frac12.
\end{equation*}
\end{proposition}
In the proof of this proposition, we will make use of the following Lemma (compare~\cite{Spiegelhofer2022}).
\begin{lemma}\label{lem_cot}
Let $\bft\in\NI^{2N}$.
We have
\begin{equation*}
c_{\bft} = \frac 12 + \frac{\mu_{\bft}(0)}2 +
\frac 1{4\pi}
\int_0^{2\pi}
\imagpart \gamma_{\bft}(\vartheta)\cot(\vartheta/2)\,\mathrm d\vartheta,
\end{equation*}
where the integrand is a bounded function.
\end{lemma}

We also repeat~\cite[Lemma~2.2]{Spiegelhofer2022}, and adapt the proof of Proposition~2.1 in that paper.
\begin{lemma}
For $j\geq 1$ we have
\begin{equation}\label{eqn_integral_identity}
\int_0^{2\pi}
\sin(j\vartheta)\cot(\vartheta/2)\,\mathrm d\vartheta
=2\pi,
\end{equation}
where the integrand is bounded.
\end{lemma}
\begin{proof}[Proof of Lemma~\ref{lem_cot}]
Let $\varepsilon>0$, and choose $m$ so large that
\[\sum_{j\ge m}\mu_{\bft}(j)<\varepsilon.\]
It follows from~\eqref{eqn_Delta_def}) that
there exists an error term $E$ such that $\lvert E\rvert<\varepsilon$
and
\begin{equation}\label{eqn_ct_summation}
c_{\bft}
=\sum_{0\leq j<m}\mu_{\bft}(j)
+E
=\int_0^1
\gamma_{\bft}(\vartheta)
\sum_{0\leq j<m} \e(-j\vartheta)\,\mathrm \,\mathrm d\vartheta
+E
=
\int_0^1
\realpart
\gamma_{\bft}(\vartheta)
\frac{1-\e(-m\vartheta)}{1-\e(-\vartheta)}
\,\mathrm d\vartheta+E.
\end{equation}

Consider the product representation~\eqref{eqn_Mt_def},~\eqref{eqn_gamma_blocks_infinite} of $\gamma_{\bft}(\alpha,\beta)$,
as well as the definitions of $A_0^\ell$, $A_1^k$ for $\ell,k\in\NI$.
We see that
bivariate polynomials in $\alpha,\beta$ appear in the numerators of each of the four matrix entries, and denominators of the form $(1-\alpha)^K(1-\beta)^L$, \emph{where $K,L\leq N$}.
Expanding these formal power series, a relation
\[\bigl[\alpha^i\beta^j]\gamma_{\bft}(\alpha,\beta)\leq C(i+1)^N(j+1)^N
\]
follows (where the constant may depend on $\bft$).
In particular, \emph{all moments of $\mu_{\bft}$ exist},
as each of the expressions
\[
m_L(\bft)
\eqdef
\sum_{a\in\mathbb Z}
a^L
\mu_{\bft}(a)
=
\sum_{a\in\mathbb Z}
a^L
\sum_{i-j=L}
\frac1{2^{i+j}}
\bigl[\alpha^i\beta^j\bigr]
\gamma_{\bft}\bigl(\alpha,\beta\bigr)
\]
is finite.
Thus,
\[\imagpart\gamma_{\bft}(\vartheta)=\sum_{a\in\mathbb Z}\mu_{\bft}(a)\sin(a\vartheta)
=\LandauO(\vartheta)\]
for $\vartheta\rightarrow 0$, while
$\cot(-\pi\vartheta)=\LandauO(1/\vartheta)$.
Note also that
\begin{equation}\label{eqn_cot_introduction}
\realpart \frac 1{1-\e(-\vartheta)}=\frac 12\quad\mbox{and}\quad
\imagpart \frac 1{1-\e(-\vartheta)}=\frac 12\cot(-\vartheta/2).
\end{equation}

Equations~\eqref{eqn_ct_summation} and~\eqref{eqn_cot_introduction} imply
\begin{equation}\label{eqn_ct_expansion}
\begin{aligned}
c_{\bft}
&=
\int_0^1 
\realpart
\frac{\gamma_{\bft}(\vartheta)}{1-\e(-\vartheta)}
-
\realpart
\frac{\gamma_{\bft}(\vartheta)\e(-m\vartheta)}{1-\e(-\vartheta)}
\,\mathrm d\vartheta+E
\\&=
\int_0^1 
\realpart
\frac{\gamma_{\bft}(\vartheta)}{1-\e(-\vartheta)}
-\frac 12\realpart \bigl(\gamma_{\bft}(\vartheta)\e(-m\vartheta)\bigr)
\\&\hspace{10em}+
\frac 12\imagpart \bigl(\gamma_{\bft}(\vartheta)\e(-m\vartheta)\bigr)\cot(-\pi\vartheta)
\,\mathrm d\vartheta+E
.
\end{aligned}
\end{equation}
All occurring summands are bounded functions.
Since $\sum_{j<m}\mu_{\bft}(j)\le1$, it follows that
\[\bigl\lvert \gamma_{\bft}(\vartheta)\e(-m\vartheta)\bigr\rvert
\leq \biggl\lvert\sum_{\ell\geq 1}a_\ell \e(-\ell\vartheta)\biggr\rvert+\varepsilon
\]
for some nonnegative $a_\ell$ such that $\sum_{\ell\geq 1}a_\ell\leq 1$.
The contribution of the second term in the middle line of~\eqref{eqn_ct_expansion} is therefore bounded by $\varepsilon$.
We obtain
\[c_{\bft}=\int_0^1 \realpart\frac{\gamma_{\bft}(\vartheta)}{1-\e(-\vartheta)}
+\frac 12\sum_{\ell\geq 1}
a_\ell\sin(-2\pi\ell\vartheta)\cot(-\pi\vartheta)\,\mathrm d\vartheta.    \]
The partial sums $\sum_{1\leq \ell<L}a_\ell\sin(-2\pi\ell\vartheta)\cot(-\pi\vartheta)\,\mathrm d\vartheta$ are bounded, uniformly in $L$, by an integrable function on $[0,1]$.
Interchanging the summation and the integral, applying the identities (\ref{eqn_integral_identity}) and~\eqref{eqn_cot_introduction},
and letting $\varepsilon\to0$, the statement follows.
\end{proof}

\begin{proof}[Proof of Proposition~\ref{prp_dominant_component}]
In Corollary~\ref{cor_c_explicit}, the only summand that remains satisfies
$\bftau=(\tO,\ldots,\tO)\in\NI^{2m}$, which leads to the identity
$\DeltaC{\bft}{0}=c_{\bft}(0)$.
By Theorem~\ref{thm_components}, the function $\gamma_{\bft}$ satisfies
\[\gamma_{\bft}(\vartheta)=\compPrefactor_{N,0}(\vartheta)=\frac1{(5-4\cos \vartheta)^N},\]
and is therefore real-valued and symmetric.
Lemma~\ref{lem_cot} yields
\[c_{\bft}=\frac12+\frac{\mu_{\bft}(0)}{2},\]
where $\mu_{\bft}(0)$ is the Fourier coefficient
\begin{equation}\label{eqn_Fourier_coefficient}
\mu_{\bft}(0)=
\hat \gamma_{\bft}(0)=\frac1{2\pi}\int_0^{2\pi}
\frac{\mathrm d\vartheta}{(5-4\cos\vartheta)^N}>0.
\end{equation}
\end{proof}

The first identities produced by this proposition are
{\everymath={\displaystyle}
\begin{equation}\label{eqn_rsum_first_values}
\begin{array}{lllr}
\rsum_1(0)&=\lim_{k\rightarrow\infty}c_{2^k-1}&=&2/3,\\[2mm]
\rsum_2(0)&=\lim_{k\rightarrow\infty}c_{(2^k-1)(2^{2k}+1)}&=&16/27,\\[2mm]
\rsum_3(0)&=\lim_{k\rightarrow\infty}c_{(2^k-1)(2^{4k}+2^{2k}+1)}&=&138/243,\\[2mm]
\rsum_4(0)&=\lim_{k\rightarrow\infty}c_{(2^k-1)(2^{6k}+2^{4k}+2^{2k}+1)}&=&1216/2187,\\[2mm]
\rsum_5(0)&=\lim_{k\rightarrow\infty}c_{(2^k-1)(2^{8k}+2^{6k}+2^{4k}+2^{2k}+1)}&=&10802/19683.
\end{array}
\end{equation}
}
This is consistent with the formulas given after Remark~\ref{rem_canonical_reduction}.
The sequence defined by Proposition~\ref{prp_dominant_component} apparently can be found in the OEIS~\cite[\texttt{A369580}]{OEIS}:
the identity
\begin{equation}\label{eqn_OEIS_A369580}
\rsum_m(0)
=A_{m,m}
\end{equation}
seems to be satisfied,
where the right hand side is the diagonal of the array $A$ defined by
{\everymath={\displaystyle}
\begin{equation*}
\begin{array}{rll}
A_{k,0}&=1&\mbox{for }k\ge0,\\
A_{0,\ell}&=0&\mbox{for }\ell\ge1,\\
A_{k,\ell}&=\frac{A_{k-1,\ell}+A_{k,\ell-1}+A_{k-1,\ell-1}}3&\mbox{for }k,\ell\ge1,
\end{array}
\end{equation*}
}
see Figure~\ref{fig_diagonal}.
  \begin{figure}[!ht]
{\everymath={\displaystyle}
\[
\begin{array}{rrrrrrrr}
1&\tfrac01&\tfrac03&\tfrac09&\tfrac0{27}&\tfrac0{81}\\[2mm]
\tfrac11&\tfrac23&\tfrac29&\tfrac{2}{27}&\tfrac{2}{81}&\tfrac{2}{243}\\[2mm]
\tfrac33&\tfrac89&\tfrac{16}{27}&\tfrac{24}{81}&\tfrac{32}{243}&\tfrac{40}{729}\\[2mm] 
\tfrac99&\tfrac{26}{27}&\tfrac{66}{81}&\tfrac{138}{243}&\tfrac{242}{729}&\tfrac{378}{2187}\\[2mm]
\tfrac{27}{27}&\tfrac{80}{81}&\tfrac{224}{243}&\tfrac{560}{729}&\tfrac{1216}{2187}&\tfrac{2320}{6561}\\[2mm]
\tfrac{81}{81}&\tfrac{242}{243}&\tfrac{706}{729}&\tfrac{1938}{2187}&\tfrac{4834}{6561}&\tfrac{10802}{19683}
\end{array}
\]
}
\caption{An array containing $\rsum_0(0),\ldots,\rsum_5(0)$ as its diagonal}
\label{fig_diagonal}
\end{figure}
We expect that proving $\rsum_k(0)=A_{k,k}$ for all $k$ is a straightforward application of the Weierstrass substitution $u=\tan(\vartheta/2)$ on the integral~\eqref{eqn_Fourier_coefficient}.


To demonstrate Corollary \ref{cor_c_explicit}, we now perform explicit calculations of $c_t$, for all non-negative integers $t$ having $N=1,2$ maximal blocks of $\tL$s in its binary expansion. In particular, we are going to show that Cusick's conjecture holds for these cases.

\subsection{The case $N=1$}
Let $t$ have binary expansion $\tL^k$, i.e., $t=2^k-1$. Then its extended binary expansion is
 $\mathbf{t}=(k,\infty) \in \NI^2$ and we have $\teven=(k,0),\todd=(0,\infty)$. 
The formula \eqref{eqn_Delta_kl_R} comprises two nonzero summands, corresponding to $\bftau \in \{\tO\tO,\tL\tO  \}$, and becomes 
\begin{align*} c_{\bft} &= \mathcal{D}_0(0) - 2^{-k} \mathcal{D}_1(-k) \\
&= \rsum_1(0) - 2^{-k}\left(\rsum_1(-k)-\rsum_0(-k) \right) \\
&= \frac{2}{3} - 2^{-k}\left(1-\frac{2^{-k}}{3}-1\right) \\
&= \frac{2}{3} + \frac{1}{3 \cdot 4^k},
\end{align*}
which is bounded from below by $2/3$. Hence, Cusick's conjecture holds in this case, and furthermore the values $c_t$ decrease as the length of the block $\tL^k$ increases.
 
\subsection{The case {$N=2$}}
Consider $t$ with usual binary expansion $\tL^k \tO^\ell \tL^m$, where $k,\ell,m \geq 1$. We are now going to prove that $c_t > 5/9$, which will follow from the chain of inequalities
\begin{equation} \label{eq:c_inequalities}
    c_{(m,\ell,k,\infty)} > c_{(m,\ell,\infty,\infty)} \geq c_{(\infty,\ell,\infty,\infty)} \geq c_{(\infty,1,\infty,\infty)}= \frac{5}{9}.
\end{equation}
We note that it should be possible to directly obtain the inequality between outermost terms. However, this step-by-step method gives more insight and illustrates a more general phenomenon, discussed in Section \ref{sec_conjectures}.
Let 
$$\bft_1 = (m,\ell,k,\infty), \quad \bft_2 = (m,\ell,\infty,\infty),  \quad  \bft_3 = (\infty,\ell,\infty,\infty),  \quad  \bft_4= (\infty,1,\infty,\infty)$$
denote extended binary expansions appearing in \eqref{eq:c_inequalities}. 
Again, the only $\bftau$ contributing to $c_{\bft_1}$ are those with a $\tO$ on the last position. Hence, from \eqref{eqn_Delta_kl_R} we get the following expression, where subsequent terms correspond to such $\bftau$ ordered lexicographically:
\begin{align*} c_{\bft_1} = &\mathcal{D}_0(0)-2^{-m}\mathcal{D}_1(-m)-2^{-\ell}\mathcal{D}_1(\ell) +2^{-(\ell+m)}\mathcal{D}_1(\ell-m)-2^{-k}\mathcal{D}_1(-k) \\&+2^{-(k+m)} 
\mathcal{D}_2(-(k+m))+2^{-k-\ell}\mathcal{D}_1(-k + \ell)-2^{-(k+\ell+m)}\mathcal{D}_1(-k+\ell-m). 
\end{align*}
Now, the formula for $c_{\bft_2}$ is obtained by passing $k \to \infty$, which sets all summands containing $k$ to $0$. Hence, we have
$$ c_{\bft_1} - c_{\bft_2} = 2^{-k}\mathcal{D}_1(-k) +2^{-(k+m)} 
\mathcal{D}_2(-(k+m))+2^{-k-\ell}\mathcal{D}_1(-k + \ell)-2^{-(k+\ell+m)}\mathcal{D}_1(-k+\ell-m). $$
After substituting the expressions for $\rsum_m(a)$ and performing some algebraic manipulation, we finally arrive at the formula
\begin{equation*}
c_{\bft_1} - c_{\bft_2} =\frac{1}{27} \begin{cases}
  4^{-k-m} \left( 3k(4^m-1) +2 \cdot 4^m-3 m+7\right)+4^{-l} \cdot 6m  &\text{if }  k < \ell, k+m < \ell, \\
  4^{-k-m} \left(4^m(3k+2)-3 l+9\right) + 2 \cdot 4^{-l} (3 l-3 k-1)&\text{if }  k < \ell, k+m \geq \ell, \\
  3 \cdot 4^{-k-m} \left( \ell\left(4^m-1\right)+3\right) &\text{if }  k \geq \ell.     
\end{cases} 
\end{equation*}
In each case the obtained expression is positive.

The inequality $c_{\bft_2} \geq c_{\bft_3}$ follows by passing to the limit with $m$ in the inequality we have just proved, and using the digit reversal property (see \cite{MorgenbesserSpiegelhofer2012} or Corollary \ref{cor_reversal}).

Furthermore, we have
\begin{align*} c_{\bft_3} - c_{\bft_4} &= -2^{-\ell} \mathcal{D}_1(\ell) + 2^{-1} \mathcal{D}_1(1) \\
&=-2^{-\ell}(\rsum_2(\ell)-\rsum_1(\ell))+\frac{1}{2}(\rsum_2(1)-\rsum_1(1)) \\
&= \frac{1}{27 \cdot 4^\ell}  \left(4^\ell-6 \ell+2\right) \geq  0.
\end{align*}

Finally, we get
$$c_{\bft_4} = \mathcal{D}_0(0) - \frac{1}{2} \mathcal{D}_1(1) = \rsum_2(0) - \frac{1}{2} (\rsum_2(1)-\rsum_1(1)) =\frac{5}{9}.$$

\subsection{General $N$}
As we have seen, the calculations may quickly become tedious even for very small $N$. Possible approaches to this problem are discussed in Section \ref{sec_conjectures}. For now, in the remainder of this section we give a brief description of $c_\bft$ for general $N$.

Consider the formula \eqref{eqn_Delta_kl_R} for $c_\bft$. 
It is a sum of basic terms of the form $C \cdot 2^{-\bftau \cdot \bft} \rsum_m (\bftau \cdot (\todd-\teven))$,
where $C$ is a constant and
$$  \rsum_m(a) =\sum_{i-j\ge a}
\frac1{2^{i+j}}
\bigl[\alpha^i\beta^j\bigr]
\biggl(\frac{\beta\alpha}{(1-\alpha)(1-\beta)}
\biggr)^m. $$  
As was the case with $\rsum_m$ for $m=0,1,2,3,4$ given explicitly earlier in this section, in general we have
$$
\rsum_m(a) = \begin{cases}
    2^a P^+_m(a) &\text{if } a>0, \\
    1- 2^{-a} P^-_m(a) &\text{if } a\leq 0,
\end{cases}
$$
where $P^+_m, P^-_m$ are certain polynomials of degree $m-1$. Since in our case the values $a$ are of the form $\bftau \cdot (\todd-\teven)$, we get a piecewise formula for $c_\bft$, where the domain $\NI^{2N}$ is divided into parts generated by the sets
 \begin{equation*}
    \{\bft: \bftau \cdot \teven < \bftau \cdot \todd \}, \quad \bftau \in \{\tO, \tL\}^{2N}.
 \end{equation*}
 We have already seen this when considering $N=2$, although to a lesser extent due to taking the difference
$c_{\bft_1}-c_{\bft_2}$ (which led to the cancellation of some terms).

Now, by Theorem \ref{thm_SW23} it is only necessary to prove Cusick's conjecture $N < N_0$, where $N_0$ is effective. However, due to the behavior described above (i.e. huge number of cases) we run into computational difficulties, even for relatively small $N$. Further still, even if we were to obtain a full expression for $c_\bft$ in each case, we would need to verify that for all possible $\bft \in \N$ we have $c_\bft > 1/2$. Clearly, if we were to take the computational approach for small $N$, further refinement of the method is required. In Section \ref{sec_conjectures} below we discuss some possibilities.

\section{Further results and conjectures} \label{sec_conjectures}

In this section we explore some possible ways to make further progress by dividing the problem into smaller parts. Roughly speaking, we analyze the effect on $c_\bft$ of the following operations:
\begin{enumerate}
    \item changing lengths of blocks of $\tO$s and $\tL$s;
    \item appending new blocks of digits.
\end{enumerate}
For simplicity we focus on the case where $\ell_{N-1} = \infty$, which includes all odd $t \in \N$. 
\subsection{Changing block lengths}

Based on explicit computations in Section \ref{sec:simplest_cases} and numerical experiments, we believe it may be fruitful to study the effect of prolonging blocks of $\tL$s and shortening blocks of $\tO$s in the binary expansion. Here and in the sequel, for a symbol $a$ and $n \geq 1$ we write $(a)^n = (a,a, \ldots, a)$, where $a$ is repeated $n$ times. To better distinguish the lengths of blocks of $\tL$s and $\tO$s we use the notation $(k_0, \ldots k_{N-1}) \ast (\ell_0, \ldots \ell_{N-1}) = (k_0, \ell_0, \ldots k_{N-1}, \ell_{N-1})$, as in \eqref{eqn_interleaving_def}.

We state a conjecture which says that among $\bft \in \NI^{2N}$ with $\ell_{N-1} = \infty$, a lower bound for $c_\bft$ is achieved in the extreme case:
$$  \bft_N \eqdef (\infty)^N \ast ((1)^{N-1},\infty) = (\infty,1,\ldots \infty,1,\infty,\infty),  $$
namely for blocks of $\tL$s having length $\infty$ and blocks of $\tO$ --- length $1$, yields .

\begin{conjecture} \label{con:c_ineq}
For all $\bft =(k_0, \ell_0,\ldots,k_{N-1},\ell_{N-1})\in \NI^{2N}$ such that $\ell_{N-1} = \infty$ we have
    \begin{equation} \label{eq:c_ineq}
        c_\bft \geq  c_{\bft_N}.
    \end{equation}  
\end{conjecture}

We have verified that this conjecture holds for the extended binary expansion $\bft$ of each positive integer $t<2^{20}$.
Clearly, it is only useful (in proving Cusick's conjecture) if the right-hand side of \eqref{eq:c_ineq} is strictly greater than $1/2$. 
By Theorem \ref{thm_SW23}, it suffices to check this inequality for $N < N_0$, where $N_0$ is an effective constant whose value can be extracted from the proof of the theorem. This is a quite tedious task, which likely yields a value $N_0$ too large to perform the computations on a standard computer.
Instead, we use a more elegant approach and give an exact expression for $c_{\bft_N}$ in terms of a finite sum, which yields a simple lower bound $>1/2$, valid for any $N \in \N$. First, we state two standard lemmas concerning binomial coefficients.

\begin{lemma} \label{lem:binomial_identities}
For all $m,n \in \N$ and $x \in \C \setminus \{1\}$ we have
    \begin{equation} \label{eq:identity1}
    \sum_{0 \leq k \leq n} \binom{m+k}{k} x^k = \frac{1}{(1-x)^{m+1}} - \frac{x^{n+1}}{1-x}  \sum_{0 \leq j \leq m} \binom{m+n+1}{m-j} \left( \frac{x}{1-x}\right)^j.
\end{equation} 
In particular, we have
\begin{equation} \label{eq:identity2}
    \sum_{0 \leq k \leq m}  \binom{m+k}{k} \frac{1}{2^k} = 2^m
\end{equation}
and 
\begin{equation} \label{eq:identity3}
    \sum_{0 \leq k \leq m}  \binom{m+k}{k} \frac{1}{4^k} = \left( \frac{4}{3}\right)^{m+1} - \frac{1}{3} \cdot \frac{1}{12^m} \sum_{0 \leq j \leq m} \binom{2m+1}{j} 3^j.
\end{equation}
\end{lemma}
\begin{proof}
  The formula \eqref{eq:identity1} follows from $m$-fold differentiation of the function 
$$  \frac{x^m}{m!} \sum_{0 \leq k \leq n} x^k = \frac{x^m}{m!} \cdot  \frac{x^{n+1}-1}{x-1}.$$
Plugging in $n=m$ and $x=1/2, x=1/4$, we get \eqref{eq:identity2}, \eqref{eq:identity3}.
\end{proof}

The next lemma can be quickly derived from the famous Wallis' formula for $\pi/2$ (see \cite{Hirschhorn2015} for a proof).

\begin{lemma} \label{lem:central_binomial_bounds}
For all $n \geq 1$ we have
$$ \frac{1}{\sqrt{\pi(n+1/2)}} < \frac{1}{4^n} \binom{2n}{n} < \frac{1}{\sqrt{\pi n}}.$$
\end{lemma}

We are now ready to compute an exact expression for $c_{\bft_N}$ and a simple lower bound for this value.

\begin{proposition} \label{prop:c_infty}
    For all $N \in \N$ we have
    \begin{equation} \label{eq:c_infty}
        c_{\bft_N} = \frac{1}{2} +\frac{1}{4^N} \binom{2N-2}{N-1} - \frac{1}{12^N} \sum_{0 \leq j \leq N-1} \binom{2N-1}{j} 3^j. 
    \end{equation}
Moreover, the following inequality holds:
$$c_{\bft_N} >  \frac{1}{2} + \frac{1}{4 \sqrt{\pi}} N^{-3/2}.$$
\end{proposition}
\begin{proof}
    We have
    $$  \gamma_{\bft_N}(\alpha,\beta) = 
    \begin{pmatrix} 1&0 \end{pmatrix} \bigl(A_1^\infty  A_0 \bigr)^{N-1} A_1^\infty
        \begin{pmatrix} 1 \\ \frac{\alpha}{1-\beta} \end{pmatrix} 
    =\left( \frac{\beta}{1-\alpha} \right)^N \frac{\alpha}{1-\beta}.  $$
    Hence,
    \begin{align*}
        \gamma_{\bft_N}\left(\frac{\e(\vartheta)}{2},\frac{\e(-\vartheta)}{2} \right) &=  \frac{\e((1-N)\theta)}{2^{N+1}} \sum_{j,k \geq 0} \frac{\e((k-j) \theta)}{2^{k+j}} \binom{N+k-1}{k} \\
        &= \frac{\e((1-N)\theta)}{2^{N+1}} \sum_{m \in \Z} 2^m \e(m \theta) \sum_{k \geq  \max\{m, 0 \}} \frac{1}{4^k} \binom{N+k-1}{k} \\
        &= \e((1-N)\theta) \sum_{m \in \Z} a_m \e(m \theta),
    \end{align*} 
    where
    $$ a_m =  2^{m-N-1} \left( \left(\frac{4}{3}\right)^N - \sum_{0 \leq k < \max\{m, 0 \}} \frac{1}{4^k} \binom{N+k-1}{k} \right).$$
    We thus get
   \begin{align}
       c_{\bft_N} &= \sum_{m \geq N-1} a_m = 1 - \sum_{-\infty \leq m < N-1} a_m \nonumber \\
       &= 1 - \frac{1}{4} \cdot \left(\frac{4}{3}\right)^N  + \frac{1}{2^{N+1}} \sum_{0 \leq m < N-1} 2^m \sum_{0 \leq k < m} \frac{1}{4^k} \binom{N+k-1}{k}. \label{eq:c_infty2}
   \end{align}
   Switching the order of summation, the double sum can be rewritten as
   \begin{align}
   &\sum_{0 \leq k < N-2}  \frac{1}{4^k} \binom{N+k-1}{k} (2^{N-1}-2^{k+1})  = \nonumber \\
   &\frac{1}{2^{N-1}} \binom{2N-2}{N-1} + 2^{N-1} \sum_{0 \leq k \leq N-1} \frac{1}{4^k} \binom{N+k-1}{k} -2 \sum_{0 \leq k \leq N-1} \frac{1}{2^k} \binom{N+k-1}{k} = \label{eq:double_sum}\\
   & \frac{1}{2^{N-1}} \binom{2N-2}{N-1} +2^{N-1} \left(  \left( \frac{4}{3}\right)^{N} - \frac{1}{3} \cdot \frac{1}{12^{N-1}} \sum_{0 \leq j \leq m} \binom{2N-1}{j} 3^j\right)- 2^N, \nonumber
   \end{align}
   where we have applied the identities \eqref{eq:identity2} and \eqref{eq:identity3} to $m=N-1$.
  Plugging this into \eqref{eq:c_infty2}, after some cancellation we finally reach \eqref{eq:c_infty}.

   In order to get a lower bound, we extract the final summand in \eqref{eq:c_infty} and bound the remaining terms by a geometric progression
\begin{align*}
   \sum_{0 \leq j \leq N-1} 3^j\binom{2N-1}{j} &\leq  3^{N-1} \binom{2N-1}{N-1} + \binom{2N-1}{N-2} \sum_{0 \leq j \leq N-2} 3^j \\
    &<  3^{N-1} \binom{2N-1}{N-1} + \frac{ 3^{N-1}}{6}  \binom{2N-1}{N-2} = 3^{N-1}\left(\frac{N}{1+N}-\frac{1}{6N} \right)\binom{2N-2}{N-1} \\
    &\leq 3^{N-1}\left(1-\frac{1}{N} \right)\binom{2N-2}{N-1},
\end{align*}
where the last inequality holds for $N \geq 5$.

Using Lemma \ref{lem:central_binomial_bounds}, for $N \geq 5$ we finally obtain
\begin{equation} \label{eq:c_infty_ineq}
    c_\bft > \frac{1}{2} + \frac{1}{4^N\cdot N} \binom{2N-2}{N-1} > \frac{1}{2} + \frac{1}{4 N \sqrt{\pi (N-1/2) }} > \frac{1}{2} +\frac{1}{4 \sqrt{\pi}} N^{-3/2}.
\end{equation}
We can check numerically that the inequality between outermost terms also holds for $1 \leq N < 5$, which ends the proof.
\end{proof}

\begin{remark}
The lower bound $1/2+N^{-3/2}/(4\sqrt\pi)$ should be compared to~\cite[Equation~2.30]{SpiegelhoferWallner2023}.
It implies that
\begin{equation}\label{eqn_SW2023}
c_t>1/2+CN^{-3/2}
\end{equation}
for all $t$ having $N$ blocks in binary, where $N$ is large enough, and $C$ is some absolute constant.
The requirement that $N$ be large is a significant restriction, not present in~\eqref{eq:c_infty_ineq}.
Theorem~2 in~\cite{DrmotaKauersSpiegelhofer2016} is of a similar nature.
It states that the sequence of integers $t_j$ having binary expansion $(\tL\tO)^{j-1}\tL$ satisfies
\begin{equation}\label{eqn_DKS2016}
c_{t_j}=\frac12+\frac{\sqrt{3}}{4\sqrt{2\pi}}j^{-1/2}+\LandauO\bigl(j^{-3/2}\bigr).
\end{equation}
\end{remark}

As already discussed, we get the following corollary. 

\begin{corollary}
   If Conjecture \ref{con:c_ineq} holds for all $N \leq N_0$, then Cusick's conjecture is true.
\end{corollary}

\begin{remark}
    With additional effort it is possible to obtain a tighter lower bound by extracting $m$ final terms (instead of just one) from the sum in \eqref{eq:c_infty}.    
\end{remark}

Now, numerical results suggest that it may be possible to refine Conjecture \ref{con:c_ineq} by inserting intermediate inequalities into $c_{\bft} \geq c_{\bft_N}$. More precisely, these are obtained by changing one block length at a time, setting $k_j = \infty$ or $\ell_j=1$. We have already seen an example of this in the case $N=2$, where we have obtained the chain of inequalities \eqref{eq:c_inequalities}.

In the considered case $\ell_{N-1} = \infty$ we note that not all ``paths'' leading from $c_\bft$ to $c_{\bft_N}$ are viable. In our experiments, we have observed that an ``incorrect'' inequality can occur when a block of $\tO$s of length $2$ is shortened to length $1$, as in the following example  (the modified block is written in bold):
$$     c_{(1,2,1)\ast(1,\mathbf{2},\infty)} \approx 0.602 < c_{(1,2,1)\ast(1,1,\infty)} \approx 0.609. $$
 On the other hand, prolonging blocks of $\tL$s to $\infty$, and only then shortening the block of $\tO$s is viable: 
   \begin{align*}
     c_{(\mathbf{1},2,1)\ast(1,2,\infty)} &> c_{(\infty,\mathbf{2},1)\ast(1,2,\infty)} > c_{(\infty,\infty,\mathbf{1})\ast(1,2,\infty)} \\
     &> c_{(\infty,\infty,\infty)\ast(1,\mathbf{2},\infty)}  >   c_{(\infty,\infty,\infty)\ast(1,1,\infty)}.
 \end{align*} 

 Furthermore, it seems that values $c_\bft$ are usually decreasing in $k_j$ and increasing in $\ell_j$. However, this is not always the case, as seen above for blocks of $\tO$s. For blocks of $\tL$s a counterexample is given by $\bft= (3,2,2,2,1) \ast (1,1,1,3,\infty)$, 
 where increasing $k_4=1$ to $2$ yields a larger value $c_\bft$. We refrain from formulating a detailed conjecture concerning monotonicity, as one would need to impose certain restrictions on block lengths, based on more extensive numerical experiments.

We now take a look at one of the simplest cases in Conjecture \ref{con:c_ineq}, namely the inequality
$c_\bft \geq c_{\bft_N}$, where $\bft = (k, (\infty)^{N-1}) \ast ((\tL)^{N-1},\infty)$ and $k \in \N$ is a variable. Writing this in the form of a binary expansion, we would like to know whether prolonging the final block of $\tL$s in  $\tO^\infty \tL^\infty \tO \cdots \tL^\infty \tO \tL^k$ to infinite length yields a smaller value $c_{\bft}$. 
Using the definition of $\gamma_\bft$, we get
\begin{align*}
    \gamma_{\bft} - \gamma_{\bft_N}  &= \begin{pmatrix}
    1 & 0
\end{pmatrix}  
(A_1^k-A_1^\infty) A_0 \left( A_1^\infty A_0 \right)^{N-1} \begin{pmatrix}
    1 \\ 0
\end{pmatrix} \\
&= \alpha^{k+1}  \frac{1-\alpha-\beta}{(1-\alpha)(1-\beta)} \left( \frac{\beta}{1-\alpha}\right)^{N-1} \\
&= \alpha^{k+1} \beta^{N-1} \left(1- \frac{\alpha}{1-\beta} \right) \frac{1}{(1-\alpha)^N}.
\end{align*}  
Letting $\alpha=\e(\vartheta)/2, \beta=\e(-\vartheta)/2$, this expression becomes
$$\frac{\e((k-N+2)\vartheta)}{2^{k+N}} \left(1- \sum_{m \geq -1} \frac{\e(-m \vartheta)}{2^{m+2}}  \right) \sum_{j \geq 0} \frac{1}{2^j} \binom{N+j-1}{j} \e(j\vartheta)
    = \frac{\e((k-N+2)\vartheta)}{2^{k+N}} \sum_{j \in \Z} g_j \e(j\vartheta),
$$  
for suitable coefficients $g_j$. Substituting $q = N-k-2$, we would like to show that for any $q \leq N-3$ the sum $\sum_{j \geq q} g_j$ is nonnegative.

Starting with the case $q \leq 0$, we get
$$ \sum_{j \geq q} g_j = \sum_{j \geq 0} \frac{1}{2^j}  \binom{N+j-1}{j} \left(1- \sum_{-1 \leq m \leq j-q} \frac{1}{2^{m+2}} \right)  = 2^{q-2} \sum_{j \geq 0} \frac{1}{4^j} \binom{N+j-1}{j} = 2^{q-2} \left(\frac{4}{3} \right)^N.$$
If $q \geq 1$, then we have
\begin{align} \label{eq:g_sum}
    \sum_{j \geq q} g_j &= \sum_{j \geq q} \frac{1}{2^j}  \binom{N+j-1}{j} - \sum_{j \geq q-1} \frac{1}{2^j}  \binom{N+j-1}{j} \sum_{-1 \leq m \leq j-q} \frac{1}{2^{m+2}} \nonumber \\
    &=  2^{q-2} \sum_{j \geq q}\frac{1}{4^j} \binom{N+j-1}{j} - \frac{1}{2^q} \binom{N+q-2}{q-1}. 
\end{align} 
By identity \eqref{eq:identity1} applied to $n=q-1, m=N-1,x=1/4$, we get 
$$  \sum_{j \geq q}\frac{1}{4^j} \binom{N+j-1}{j} = \left( \frac{4}{3} \right)^N- \sum_{j \leq q-1}\frac{1}{4^j} \binom{N+j-1}{j} = \frac{1}{4^{q-1}} \sum_{0 \leq j \leq N-1} \frac{1}{3^{N-j}} \binom{N+q-1}{j}.$$
Plugging this into  \eqref{eq:g_sum}, our task reduces to proving that for any integers $q,N$ such that $1 \leq q \leq N-3$ we have
\begin{equation} \label{eq_binom_ineq}
    \frac{1}{3^N}\sum_{0 \leq j \leq N-1} 3^j \binom{N+q-1}{j} - \binom{N+q-2}{q-1} > 0.
\end{equation}   
Letting $a_{N,q}$ denote the left-hand side of this inequality, we get the recurrence relation
$$  a_{N,q}=a_{N,q-1} + a_{N-1,q},$$
with boundary conditions $a_{N,1} = (4/3)^N-2, a_{1,q}=-2/3$. The generating function for these terms is
$$ \sum_{N,q \geq 1} a_{N,q} x^N y^q = \frac{ 2xy(2 x-1)}{(4 x-3) (x+y-1)}. $$
By induction it suffices to only consider $q=N-3$, which reduces our initial question to the following problem.

\begin{problem}
    Prove that $a_{N,N-3} \geq 0$ for all $N \geq 4$.
\end{problem}

We do not pursue the problem further --- methods from analytic combinatorics should be applicable here. This example already showcases the difficulties that arise when trying to prove Conjecture \ref{con:c_ineq}. For general $\bft \in \NI^{2N}$ we would likely obtain inequalities similar to \eqref{eq_binom_ineq} with multiple additional parameters (block lengths).

To conclude this subsection, we take a look at a potential upper bound on $c_\bft$ for fixed $N$. In analogy to Conjecture \ref{con:c_ineq}, we expect that such a bound is obtained by shortening blocks of $\tL$s, and prolonging blocks of $\tO$s, namely for $\bft = (1)^N \ast (\infty)^N$. In the following proposition, we derive a simple formula for this value $c_\bft$.

\begin{proposition}
    For all $N \in \N$ we have
    $$  c_{(1)^N \ast (\infty)^N} = \frac{1}{2} + \frac{1}{4^N}\binom{2N-1}{N}.$$
\end{proposition}
\begin{proof}
Put $\bft = (1)^N \ast (\infty)^N$.
We have
$$ \gamma_\bft(\alpha,\beta) = \begin{pmatrix} 1&0 \end{pmatrix} \bigl(A_1 A_0^\infty \bigr)^N
\begin{pmatrix} 1 \\0 \end{pmatrix} = \left(\frac{\alpha}{1-\beta} \right)^N,$$
and thus
    $$ \gamma_\bft\left(\frac{\e(\vartheta)}{2}, \frac{\e(-\vartheta)}{2} \right) = \e(N\theta) \sum_{k \geq 0} \frac{1}{2^{k+N}} \binom{N+k-1}{k} \e(-k \theta).$$
    It follows that
    $$ c_\bft = \frac{1}{2^N} \sum_{0 \leq k \leq N} \frac{1}{2^k} \binom{N+k-1}{k} = \frac{1}{2} + \frac{1}{4^N}\binom{2N-1}{N},$$
    where we have used the identity \eqref{eq:identity2}.
\end{proof}

\subsection{Appending blocks}
The following conjecture says that ``replacing the rightmost $\tL$ of an odd integer $t$ by $\tO\tL^\infty$'' does not increase the value $c_{\bft}$.
\begin{conjecture}\label{con_convolution}
For each odd integer $t\ge1$ we have
\begin{equation}\label{eqn_conj_appending111}
c_t\ge\lim_{K\to\infty}c_{2^Kt-1}.
\end{equation}
\end{conjecture}

The relevance of this conjecture can be explained appealing to~\cite{SpiegelhoferWallner2023}, which says in particular that $c_t>1/2$ as soon as $t$ contains sufficiently many blocks of $\tL$s.
Repeating the process $t\mapsto 2^Kt-1$,
Conjecture~\eqref{con_convolution} yields an integer having arbitrary many blocks of $\tL$s, and a smaller (or equal) $c_t$-value.
Theorem~1.1 in~\cite{SpiegelhoferWallner2023} then implies $c_t>1/2$.
Numerically, equality in~\eqref{eqn_conj_appending111} is attained for $t=5$ and $t=107$, and the remaining odd integers up to $2^{20}$ yield strict inequality.
Note that this conjecture concerns characteristic functions only differing by a factor
$\beta/(1-\alpha)$.
More precisely, 
using~\eqref{eqn_gamma_recurrence}, we have
\[
\begin{aligned}
\gamma_{2^Kt-1}(\alpha,\beta)&=\beta\gamma_t(\alpha,\beta)
+\alpha\gamma_{2^{K-1}t-1}(\alpha,\beta)=\cdots\\
&=\gamma_\alpha,\beta)\beta\sum_{0\leq \ell<K}\alpha^\ell
+\alpha^K\gamma_{t-1}(\alpha,\beta)\\
&=
\gamma_t(\alpha,\beta)\frac{\beta}{1-\alpha}
+\frac{-\beta^K}{1-\alpha}\gamma_t(\alpha,\beta)
+\alpha^K\gamma_{t-1}(\alpha,\beta),
\end{aligned}
\]
which converges to the bivariate power series $\gamma_t(\alpha,\beta)\frac{\beta}{1-\alpha}$ as $K\to\infty$. That is,
\[\gamma_t-\lim_{K\to\infty}\gamma_{2^Kt-1}=\breve\gamma_t,
\]
where
\begin{equation*}
\breve\gamma_t(\alpha,\beta)\eqdef
\frac{1-\alpha-\beta}{1-\alpha}\gamma_t(\alpha,\beta).
\end{equation*}
The quantity $\breve\gamma_t$ obviously satisfies the same recurrence as $\gamma_t$ itself (starting from a different power series).
In order to establish the sufficient condition~\eqref{eqn_conj_appending111}, it suffices to prove the simplified inequality
\begin{equation*}
\breve c_t\ge0,
\end{equation*}
where
\begin{equation*}
\breve c_t\eqdef
\sum_{0\leq j\leq i}
2^{-(i+j)}\bigl[\alpha^i\beta^j\bigr]
\breve\gamma_t(\alpha,\beta).
\end{equation*}
In this vein, future progress towards the solution of Cusick's conjecture might be obtained.

\begin{remark}
Similarly, ``appending $\tO\tL^\infty$'' to the right of the binary expansion seems to decrease the value $c_t$.
In terms of blocks $\bfk,\bfl \in \NI^{N}$, the statement
\begin{equation}\label{eqn_append}
c_{\bfk\ast\bfl} >  c_{(\infty,\bfk)\ast (1,\bfl)}
\end{equation}
appears to be true.
More concretely, for $t \in \N$ we expect $c_t > \lim_{k \to \infty} c_{2^k t +2^{k-1}-1}$.
We numerically checked the inequality $c_t>c_{2^{41}t+2^{40}-1}$ for all odd natural numbers $t<2^{20}$, which gives a strong indication that this holds true indefinitely.
\end{remark}
We now show that a special case of~\eqref{eqn_append} holds,
namely the case $\bft = \bft_N =  (\infty)^N \ast ((1)^{N-1},\infty)$, as denoted in the previous subsection.
(That is, ``infinite blocks of $\tL$s separated by isolated $\tO$s''.)
As a byproduct, we obtain $c_{\bft_N} \to 1/2$ when $N \to \infty$ (which follows from~\cite[Equation~2.30]{SpiegelhoferWallner2023}), a necessary condition for $c_{\bft_N}$ to be a lower bound. 

\begin{proposition}
    The sequence $(c_{\bft_N})_{N \geq 1}$ is strictly decreasing and 
    $$\lim_{N \to \infty} c_{\bft_N} = \frac{1}{2}.$$
\end{proposition}
\begin{proof}
    Using the formula \eqref{eq:c_infty}, we get
    $$ c_{\bft_N}- c_{\bft_{N+1}}= \frac{1}{4^N \cdot 2N} \binom{2N-2}{N-1} + \frac{1}{12^{N+1}} - \frac{1}{12^N} \sum_{0 \leq j \leq N-1} 3^j \left(\binom{2N-1}{j}- \frac{1}{4} \binom{2N+1}{j+1}\right).  $$
  We have
 $$       
    \binom{2N-1}{j} - \frac{1}{4} \binom{2N+1}{j+1} = \left(1-\frac{N(N+1/2)}{(j+1)(2N-j)} \right) \binom{2N-1}{j}  <  \frac{1}{2N+1}  \binom{2N-1}{j},
  $$
  where the last inequality follows from the fact that $(x+1)(2N-x)$ attains maximal value at $x=N-1/2$.
    This implies 
   \begin{align*}
       \sum_{0 \leq j \leq N-1} 3^j \left(\binom{2N-1}{j}- \frac{1}{4} \binom{2N+1}{j+1} \right) &<  \frac{3^N-1}{2(2N+1)} \binom{2N-1}{N-1} < \frac{3^N}{2N} \binom{2N-2}{N-1},
   \end{align*} 
   and the inequality $c_{\bft_N}> c_{\bft_{N+1}}$ follows.

   Passing to the limit, \eqref{eq:c_infty} implies
   $$ c_{\bft_N} < \frac{1}{2} +\frac{1}{4^N} \binom{2N-2}{N-1}.$$
   An application of Lemma \ref{lem:central_binomial_bounds} finishes the proof.
\end{proof}

\bibliographystyle{siam}
\bibliography{E}

\bigskip
\begin{center}
\begin{tabular}{c}
Technical University of Leoben,\\[0.5mm]
Franz-Josef-Strasse 18, 8700 Leoben, Austria, and\\[0.5mm]
Jagiellonian University,\\
Krak\'ow, Poland\\[0.5mm]
bartosz.sobolewski@uj.edu.pl\\[0.5mm]
ORCID iD: \texttt{0000-0002-4911-0062}
\end{tabular}
\end{center}

\smallskip
\begin{center}
\begin{tabular}{c}
Technical University of Leoben,\\[0.5mm]
Franz-Josef-Strasse 18, 8700 Leoben, Austria\\[0.5mm]
lukas.spiegelhofer@unileoben.ac.at\\[0.5mm]
ORCID iD: \texttt{0000-0003-3552-603X}
\end{tabular}
\end{center}\end{document}